\documentclass[10pt,a4paper]{article}

\usepackage{a4wide}

\usepackage{xfrac}

\usepackage{authblk}

\usepackage{amssymb, amsfonts, amsbsy, latexsym}
\usepackage{graphicx}
\usepackage{tikz}
\usepackage{color}
\usepackage{colortbl}
\usepackage[english]{babel}
\usepackage{url}
\usepackage{amsmath}
\usepackage{amsthm}

\usepackage[utf8]{inputenc}

\usepackage{titlesec}

\titleformat*{\section}{\large\bfseries}
\titleformat*{\subsection}{\normalsize\bfseries}
\titleformat*{\paragraph}{\large\bfseries}
\titleformat*{\subparagraph}{\large\bfseries}

\newtheorem{theorem}{Theorem}
\newtheorem{definition}[theorem]{Definition}
\newtheorem{proposition}[theorem]{Proposition}
\newtheorem{claim}[theorem]{Claim}
\newtheorem{remark}[theorem]{Remark}
\newtheorem{lemma}[theorem]{Lemma}
\newtheorem{example}[theorem]{Example}
\newtheorem{assumption}[theorem]{Assumption}

\newtheorem{approximation problem}[theorem]{Approximation problem}

\newcommand{\ip}[2]{\left\langle#1,#2\right\rangle}
\newcommand{\abs}[1]{\left|#1\right|}
\newcommand{\norm}[1]{\left\|#1\right\|}

\def\hf{\hat{f}}
\def\hs{\hat{s}}
\def\hh{\hat{h}}

\def\w{\omega}

\def\e{\epsilon}

\def\k{\kappa}
\def\a{\alpha}
\def\b{\beta}
\def\c{\gamma}

\def\NN{\mathbb{N}}

\def\cF{\mathcal{F}}

\def\cH{\mathcal{H}}

\def\cS{\mathcal{S}}

\def\CC{\mathbb{C}}
\def\NN{\mathbb{N}}
\def\ZZ{\mathbb{Z}}

\def\RR{\mathbb{R}}

\begin{document}

\title{Quasi Monte Carlo Time-Frequency Analysis}
\author[1]{Ron Levie}
\author[2]{Haim Avron}
\author[1,3]{Gitta Kutyniok}
\affil[1]{Department of Mathematics, Ludwig-Maximilians-Universit{\"a}t M{\"u}nchen}
\affil[2]{School of Mathematical Sciences, Tel Aviv University}
\affil[3]{Department of Physics and Technology, University of Troms{\o}}
\date{}                     
\setcounter{Maxaffil}{0}
\renewcommand\Affilfont{\itshape\small}

       


\maketitle



	
	
	

\begin{abstract}
We study signal processing tasks in which the signal is mapped via some generalized time-frequency transform to a higher dimensional time-frequency space, processed there, and synthesized to an output signal. We show how to approximate such methods using a quasi-Monte Carlo (QMC) approach. \textcolor{black}{We consider cases where the time-frequency representation is redundant, having feature axes in addition to the time and frequency axes. The proposed QMC method allows sampling both efficiently and  evenly such redundant time-frequency representations. Indeed, 1) the number of samples required for a certain accuracy is log-linear in the resolution of the signal space, and depends only weakly on the dimension of the redundant time-frequency space, and 2) the quasi-random samples have low discrepancy, so they are spread evenly in the redundant time-frequency space.} One example of such redundant representation is the localizing time-frequency transform (LTFT), where the time-frequency plane is enhanced by a third axis. This higher dimensional time-frequency space improves the quality of some time-frequency signal processing tasks, like \textcolor{black}{the} phase vocoder (an audio signal processing effect). Since the computational complexity of the QMC is log-linear in the resolution of the signal space, this higher dimensional time-frequency space does not degrade the computation complexity of the proposed QMC method. The proposed QMC method is more efficient than standard Monte Carlo methods, since the deterministic QMC sample points are optimally spread in the time-frequency space, while random samples are not.
\end{abstract}


\section{Introduction}

Recently, it was shown that Monte Carlo discretizations of some signal processing tasks based on continuous frames \textcolor{black}{show promising potential} \cite{Ours1,Ours2}. The goal of this paper is to improve the computational complexity of these Monte Carlo methods while retaining their desirable properties. This is done by replacing Monte Carlo with a quasi-Monte Carlo (QMC) discretization. We focus on signal processing using some general class of time-frequency transforms. 

The proposed QMC method is a midway between standard grid-based methods and Monte Carlo methods. On the one hand, the QMC samples have ``random like'' qualities -- in the context of this work this means that they are evenly distributed in the time-frequency space in a non-regular manner. The random like property is advantageous, since, 1) the number of samples needed for a certain accuracy is (up to a log factor) independent of the dimension of the time-frequency space, as opposed to regular grids that increase exponentially \textcolor{black}{with} the dimension, and 2) the even distribution of the sample points makes them suitable for feature extraction in time-frequency signal processing, as opposed to some regular sample schemes (see Subsection \ref{Time-frequency signal processing}). On the other hand, the QMC samples are deterministic. This allows choosing samples more optimally than the random sampling in Monte Carlo, which ultimately makes the proposed QMC method significantly more efficient than the Monte Carlo method. \textcolor{black}{The improved accuracy of QMC with respect to Monte Carlo is illustrated in Figure \ref{fig:comp}, where we compare Mote Carlo LTFT with QMC LTFT.}

Although the Monte Carlo and the QMC approaches are similar in philosophy, as a deterministic method, the theoretical machinery required for analyzing the QMC method is completely different. We base our approximation analysis on the Koksma–Hlawka (KH) inequality \cite{KH_orig} (see Theorem \ref{KoksmaHlawka}). The main challenge in using the KH inequality in our analysis comes from the fact that the error bound depends on the derivatives of the integrand. This is a problem in time-frequency analysis, since high frequency atoms are highly oscillatory with very large derivatives. We overcome this problem by showing that the time-frequency representations of time signals have certain built-in constraints that limit the magnitudes of their derivatives.

\subsection{Time-frequency signal processing}
\label{Time-frequency signal processing}

Time frequency analysis is the theory and methodology of decomposing time signals to their different local frequency components. The theory can be summarized in a general form as follows. Local frequencies are given as time-frequency atoms: signals localized in short time intervals with distinct frequencies. 
Decomposing time signals to their time-frequency content is done via the time-frequency transform, also called the analysis transform. 
On the other hand, functions in the time-frequency space can be mapped to time signals by the synthesis transform. 
The synthesis transform is typically the pseudo inverse, approximate inverse, or adjoint of the time-frequency transform.

Time-frequency signal processing is any method that decomposes a signal to its time-frequency components, manipulates these components, and recombines/synthesizes the resulting atoms to an output time signal. Some examples of time-frequency signal processing are multipliers \textcolor{black}{\cite{ex1,ex2,wave_mult0,New_mult0,New_mult1,New_mult2}}, where each time-frequency component is multiplied by a scalar that depends on the time and frequency of the atom
(with applications, for example, in audio analysis
\cite{ex3}
and improving signal to noise
\cite{ex4}), signal denoising e.g. wavelet shrinkage denoising
\cite{ex5,ex6} \textcolor{black}{and Shearlet denoising \cite{ex7}}, where the coefficient of each time-frequency atom is transformed by some non-linear scalar mapping, and phase vocoder \cite{phase_vocoder1,phase_vocoder2,vocoder_book,vocoder_imp}, where each time-frequency atom is mapped to a different time-frequency atom, and the coefficients undergo some non-linear transformation.


Two prominent examples of the general setting of time-frequency analysis are the short time Fourier transform (STFT) and the 1D continuous wavelet transform (CWT). In the STFT, the atoms are localized at time intervals of a fixed length, meaning that the higher the frequency of an atom, the more oscillations it has. In the CWT, all atoms have a fixed number of oscillations, meaning that the higher the frequency of an atom, the shorter the time interval in which it is localized. An advantage of this property of the CWT is that time-frequency atoms of equal high frequency and nearby times are separated in time due to their short time spread. This is in contrast to the STFT, where such pairs of atoms will be correlated due to their large time supports. In other words, the CWT is better at time-localizing high frequencies than the STFT, and is thus better at isolating time events. This property is important in some time-frequency signal processing tasks, which motivates us to consider \textcolor{black}{ time frequency analysis frameworks that involve the CWT} in this paper. 
However, low frequencies are represented by CWT atoms with large time supports. In \cite{Ours1}, a hybrid transform which uses STFT atoms for low and high frequencies, and CWT for middle frequencies, was proposed, namely the localizing time-frequency transform (LTFT). We focus in this paper on the LTFT.



The CWT uses the whole continuum $\RR^2$ of time-frequency pairs, and the discrete wavelet transform (DWT) uses a discrete set of time-frequency samples. 
In \cite{Ours1} it was suggested that the DWT is not appropriate for time-frequency analysis of polyphonic audio signals (audio signals that are not concentrated on a small subset of coefficients in the time-frequency plane). The basic argument is that the distance between samples in the DWT becomes exponentially large along the frequency axis, which is inappropriate for representing signals having ``uniformly'' spread time-frequency components (see Figure \ref{fig:3DTF}, left, top). Hence, while the DWT is stably invertible, its atoms are not appropriate for representing time-frequency features for signal processing tasks like phase vocoder. Thus, a different form of discretization of the CWT is required.

\subsection{Non-regular discretization of time-frequency signal processing}

In \cite{Ours1,Ours2} it was shown that choosing the time-frequency samples randomly overcomes the above problem. Such a discretization is called a Monte Carlo method. The focus of this paper is on improving the quality and computational complexity of the Monte Carlo method, by replacing it with a QMC method. 
In Figure \ref{fig:3DTF}, left, a \textcolor{black}{dyadic} wavelet grid is compared to quasi-random samples, and it is apparent that the quasi-random samples are better spread in the time-frequency plane.


The Monte Carlo method of \cite{Ours1,Ours2} also allows enhancing time-frequency analysis as follows. It is sometimes beneficial to consider a higher dimensional time-frequency space, where at each time-frequency point there is a whole space of atoms, instead of just one atom. \textcolor{black}{Namely, instead of having one atom $f_{t,\w}\in L^2(\RR)$ representing each  time-frequency point $(t,\w)\in\RR^2$, we consider a whole space $F_{t,\w}=\{f_{t,\w,c}\ |\ c\in H\}$, where $H$ is some parameter space.  All of the atoms from the space $F_{t,\w}$} represent the same time and frequency, but they differ on other properties. For example, we may enhance the CWT time-frequency space by adding a new axis $c$, which specifies the number of oscillations in the atom. Since different signal features are best represented by different time spans, where harmonic features have large time supports and percussive feature have short time supports, allowing a variety of oscillation numbers in the atoms assures that all types of features are well represented by the atom system. 
\textcolor{black}{The LTFT, introduced in \cite{Ours1,Ours2}, is based on adding the number of oscillations axis to the CWT,  where in addition, low and high frequencies are represented by STFT atoms instead of CWT atoms.}
In Figure \ref{fig:3DTF}, right, we plot the atoms of the LTFT feature space. \textcolor{black}{In general, we call a representation of the form $f_{t,\w,c}$ a \emph{redundant time-frequency representation}.}


\textcolor{black}{A common way to discretize continuous time-frequency representations is to sample the continuous system to a discrete frame $\{f_{t_n,\w_n,c_n}\}_n$, satisfying the frame inequality (e.g., as in \cite{Time_freq,Ten_lectures}).  When we use this approach to discretize a redundant time-frequency representation, the frame inequality is satisfies even if we restrict $c$ to a constant, since for each $c_0$, $\{f_{t,\w,c_0}\ |\ (t,\w)\in \RR^2\}$ is a valid continuous time-frequency representation on its own. Hence, nothing in the standard discrete frame approach requires the samples $(t_n,\w_n,c_n)_n$ to be evenly spread in $\RR^2\times H$. This is a problem if we want to represent the axis $c$ faithfully. }

\textcolor{black}{Two advantages in the Monte Carlo method are that 1) the number of samples required for a given error tolerance does not depend on the dimension of the time-frequency space, but only on the resolution of the discrete signal domain, and, 2) the Monte Carlo samples are random, and thus spread roughly evenly in the redundant time-frequency space.} Thus, in principle, we may add as many dimensions as we like to the time-frequency space without degrading computational complexity \textcolor{black}{and well-spreadness}. 
Note that this is not the case if we discretize the time-frequency space by a \textcolor{black}{discrete frame, since the size of the sample set  $(t_n,\w_n,c_n)_n$ increases linearly in the resolution along $H$}. We show in this paper that the proposed QMC method shares \textcolor{black}{the efficiency and well spreadness} properties with the Monte Carlo method, up to some weak dependency on the dimension of the time-frequency space. Moreover, the error rate of the QMC method improves that of the Monte Carlo method \textcolor{black}{(see for example Figure \ref{fig:comp})}, which overall speeds up computations since less samples are required. This makes the QMC method appropriate for redundant time-frequency analysis. 



We propose in this paper a theory for analyzing QMC discretizations of more general integral transforms, that we term general time-frequency  transforms. We show that under some assumptions, the number of samples in the time-frequency space required by our method depends only weakly on the dimension of the time-frequency space, and is mainly determined by the resolution, or dimension, of the discrete signal space.

\textcolor{black}{We note that related to \cite{Ours1,Ours2}, another line of work that randomly discretizes integral transforms is called \emph{relevant sampling} \cite{relevantS1,relevantS2,relevantS3,relevantS4}. While the goal in \cite{Ours1,Ours2} is to approximate the continuous frame with a quadrature sum, the goal in relevant sampling is to produce a stable sample set. In the context of continuous frames, this means constructing a discrete frame from the continuous frame elements.
}
\textcolor{black}{One advantage of the quadrature approach is that it is directly linked to the continuous time-frequency transform.  As noted above, the CWT is appropriate as a time-frequency feature extractor, and its quadrature approximation retains this property as an approximation. On the other hand, a discrete frame constructed from the CWT is only required to satisfy the frame inequality, not to directly approximate the continuous frame analysis and synthesis operators, and hence need not retain the properties of these continuous transform. 
}

\begin{figure}[!ht]
\centering
\includegraphics[width=0.8\linewidth]{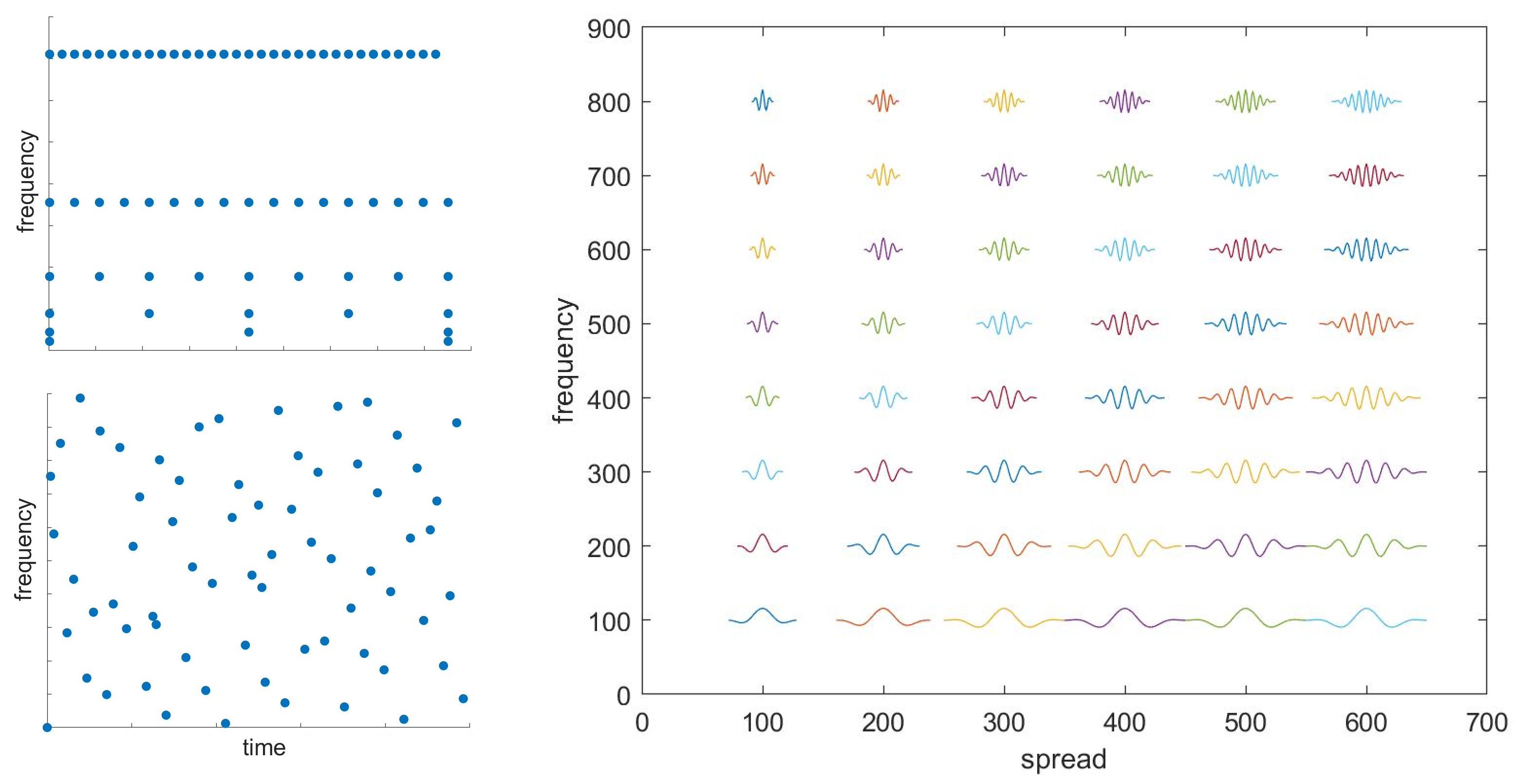}
\caption{
\emph{Left.} Top:  68 grid points of a discrete wavelet transform in the time-frequency plane. Bottom: 68 quasi-random samples of the continuous wavelet transform in the time-frequency plane. The quasi-random samples are better spread in the time-frequency plane than the grid points.
\emph{Right.} An enhanced time frequency feature space, where each local frequency is given as a wavelet atom. In addition to the time and frequency axes, a third axis determines the number of oscillations, or the spread, of each local frequency atom. Moreover, if the time spread of a wavelet atom is above some pre-defined value $S$, the wavelet atom is replaced by a STFT atom with the same frequency, and time spread $S$.  
}
\label{fig:3DTF}
\end{figure}


\subsection{Main contribution}

\textcolor{black}{In this paper, a discrete signal is an element of a finite dimensional subspace $\cS_M$, of dimension $M$, of the infinite dimensional signal space $\cH$. For example, time signals in $\cH=L^2(\RR)$ can be discretized as linear splines with knots at a predefined grid of $M$ time points. We call $M$ the resolution of the discrete signal. On the one hand, in this approach discrete signals are elements of the infinite dimensional signal space $\cH$, and can hence be analyzed in $\cH$. On the other hand, discrete signals are specified by $M$ scalars, so they are suitable to numerical analysis. We summarize our main contribution as follows.}

\begin{enumerate}
	\item 
	We consider a quasi-Monte Carlo discretization of the LTFT. The method has an error rate of 
	\begin{equation}
	O\left(\frac{M}{N}(\log N)^{d-1}\right),
	\label{eq:QMCError}
	\end{equation}
	where $M$ is the resolution of the signal space, $N$ is the number of samples, and $d=3$ is the dimension of the time-frequency space. This improves the error rate $O(\sqrt{\frac{M}{N}})$ of the Monte Carlo method of \cite{Ours1}.  The method is shown to be of computational complexity $O(N)$. When the number of samples is $N=M\log^{d-1+\e}(M)$, with any $\e>0$, the method is asymptotically accurate. 
	\item
	We consider a time dilation phase vocoder method based on the LTFT, replacing the Monte Carlo method of \cite{Ours1} with a QMC method.  In practice, for phase vocoder with time dilation $D\in\NN$, a choice of $N=4DM$ samples is sufficient for a high-quality result. 
	Sound examples and code of QMC LTFT  phase vocoder are available at \url{https://github.com/RonLevie/LTFT-Phase-Vocoder}.
	\item
	In the general case, we propose a theory for analyzing quasi-Monte Carlo discretizations of a class of integral transforms that we called general time-frequency transforms. General time-frequency transforms include the STFT, CWT, LTFT, and higher dimensional transforms like the Shearlet and Curvelet transforms. We show that in some general setting, the discretization error is given by (\ref{eq:QMCError}). 
\end{enumerate}

\begin{figure}[!ht]
\centering
\includegraphics[width=0.6\linewidth]{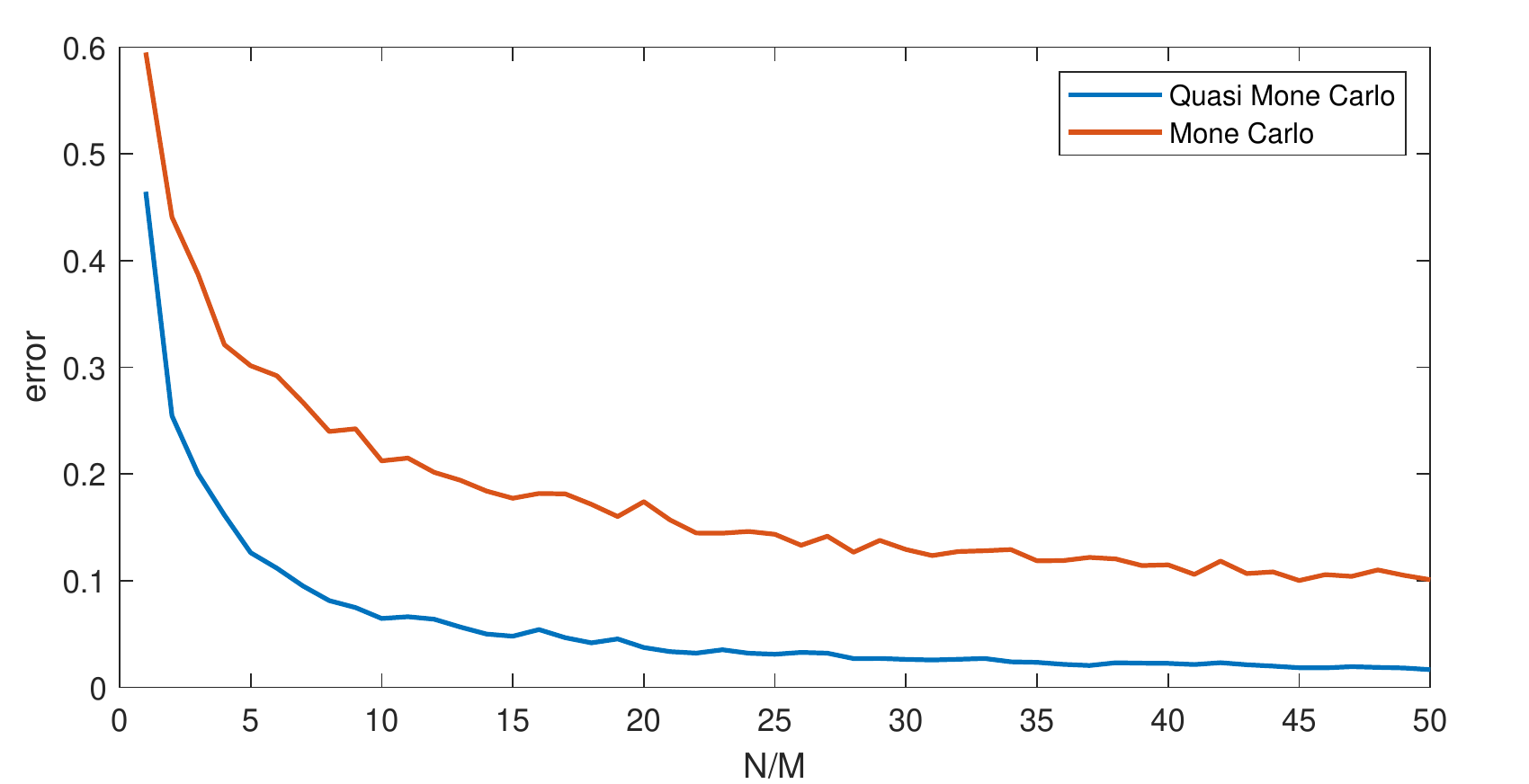}
\caption{\textcolor{black}{Reconstruction error with stochastic LTFT methods. The error is defined as the relative $L_2$ error between the original signal, and the signal after the approximate  analysis-synthesis. Red: error in the Monte Carlo method, averaged over 10 realizations. Blue: error in the QMC method. The $x$ axis is the sampling redundancy $N/M$, where $M$ is the resolution of the signal and $N$ is the number of samples in time-frequency-spread space.}
}
\label{fig:comp}
\end{figure}

\section{Background in time-frequency analysis}

In this section we present the required theoretical background for our approach. We discuss continuous frames, and how to discretize their signal spaces and restrict the coefficient space to a compact sub-domain. We define the general setting of phase space signal processing, and recall the localizing time-frequency transform. 

\subsection{Continuous frames}

The general integral transforms studied in this paper are based on continuous frames. The following definitions and claims are from \cite{Cframe1} and \cite[Chapter 2.2]{Fuhr_wavelet}, with notation adapted from the latter.

\begin{definition}
\label{CSSframe}
Let $\cH$ be a Hilbert space, and $(G,\mathcal{B},\mu)$ a locally compact topological space with $\sigma$-finite Borel measure $\mu$.
 Let $f:G\rightarrow \cH$ be a weakly measurable mapping, namely for every $s\in\cH$
\[g\mapsto \ip{s}{f_g}\]
is a measurable function $G\rightarrow\CC$.
For any $s\in\cH$, we define the \emph{coefficient function}
\begin{equation}
V_f[s]:G\rightarrow \CC \quad , \quad V_f[s](g)=\ip{s}{f_g}_{\cH}.
\label{eq:CSS2frame}
\end{equation}
\begin{enumerate}
	\item
	We call $f$ a \emph{continuous frame}, if $V_f[s]\in L^2(G)$ for every $s\in\cH$, and there exist constants $0<A\leq B<\infty$ such that
	\begin{equation}
	A\norm{s}_{\cH}^2 \leq \norm{V_f[s]}_{L^2(G)}^2 \leq B\norm{s}_{\cH}^2
	\label{eq:FB}
	\end{equation}
	for every $s\in\cH$.
	\item
	We call $\cH$ the \emph{signal space},
	 $G$ \emph{phase space}, $V_f$ the \emph{analysis operator}, and $V_f^*$ the \emph{synthesis operator}.
\item
We call $S_f=V_f^*V_f$ the \emph{frame operator}. 
\item
	We call $f$ a \emph{Parseval continuous frame}, if $V_f$ is an isometry between $\cH$ and $L^2(G)$.
\end{enumerate}
\end{definition}

The synthesis operator of a continuous frame can be computed
 by the weak integral  \cite[Theorem  2.6]{Cframe1}
\begin{equation}
V_f^*(F) = \int^{\rm w}_G F(g)f_g dg.
\label{eq:inver_proj}
\end{equation}
This integral is defined by
\begin{equation}
\ip{q}{\int^{\rm w}_G F(g)f_gdg} = \int_G \overline{F(g)}\ip{q}{f_g}dg,
\label{eq:a2}
\end{equation}
where $\int^{\rm w}_G F(g)f_gdg$ denotes the vector corresponding to the continuous functional defined in the right-hand-side of (\ref{eq:a2}), whose existence is guaranteed by the Riesz representation theorem.  Such integrals are called \emph{weak vector integrals}, or Pettis integral \cite{Weak_Integral}.

The \emph{\textcolor{black}{canonical} dual frame} \cite{Cframe1} is defined to be $\tilde{f}_g=S_f^{-1}f_g$. We have the reconstruction formula $V_{\tilde{f}}^*V_f = V_f^* V_{\tilde{f}}=I.$

\subsection{Transforms associated with time-frequency analysis}
\label{Transforms associated with time-frequency analysis}

Before we recall the  well-known STFT and CWT, we first present transforms on which they are based.
We formulate translation, modulation, and dilation, and give their formulas in the frequency domain.
\begin{definition}
\label{def:TMD}
Translation by $x$ of a signal $s:\RR\rightarrow\CC$ is defined by
\begin{equation}
[\mathcal{T}(x)s](t)= s(t-x).
\label{eq:trans00}
\end{equation}
Modulation by $\w$ of a signal $s:\RR\rightarrow\CC$ is defined by
\begin{equation}
[\mathcal{M}(\w)s](t)= s(t)e^{2\pi i \w t}.
\label{eq:trans001}
\end{equation}
Dilation by $\tau$ of a signal $s:\RR\rightarrow\CC$ is defined by
\begin{equation}
[\mathcal{D}(\tau)s](t)= \tau^{-1/2}s(\tau^{-1} t).
\label{eq:trans002}
\end{equation}
\end{definition}
Let $\cF: L^2(\RR)\rightarrow L^2(\RR)$ denote the Fourier transform, \textcolor{black}{with the normalization
\[\cF s(\w) = \int_{\RR}s(t)e^{-2\pi i \w t} dt.\]
} In (\ref{eq:trans00}), the dilation parameter $\tau^{-1}$ is interpreted as a frequency multiplier by $\tau$. Indeed, if $\hs$ is concentrated about frequency $z_0$, then $\cF[\mathcal{D}(\tau^{-1})s]$ is concentrated about frequency $\tau z_0$, as is shown in the following lemma. The proof of the following lemma is direct (see for example \cite[Sections 1.2 and 10]{Time_freq}).

\begin{lemma}
\label{Transform_lemma}
Translation, modulation, and dilation are unitary operators in $L^2(\RR)$ and take the following form in the frequency domain.
\begin{enumerate}
\item
$\cF\mathcal{T}(x)\cF^* =\mathcal{M}(-x)$.
\item
$\cF \mathcal{M}(\w)\cF^*= \mathcal{T}(\w)$.
	\item 
	$\cF \mathcal{D}(\tau)\cF^*=\mathcal{D}(\tau^{-1})$.
\end{enumerate}
\end{lemma}

\subsection{The wavelet and the short time Fourier transforms}

Two prominent transforms where phase space is interpreted as the time-frequency plane are the STFT and the CWT. In these two examples, $\cH=L_2(\RR)$ is the space of time signals, and $G=\RR^2$ with the standard Euclidean measure is the time-frequency plane. 
The atoms of the STFT are defined as
\[f_{a,b} =\mathcal{T}(a)\mathcal{M}(b)f\]
where $f\in L^2(\RR)$ is some function, called the \emph{window}, which is localized in time and frequency about $0$. The STFT system is a Parseval frame if $\norm{f}_{\cH}=1$. 

The atoms of the CWT are defined as
\[f_{a,b} =\mathcal{T}(a)\mathcal{D}(b^{-1})f\]
where $f\in L^2(\RR)$ is some function, called the \emph{mother wavelet}, which is localized in time about $0$ and in frequency about $1$, with 
\begin{equation}
\int \w\abs{\hf(\w)}^2d\w <\infty.
\label{eq:addmiss}
\end{equation}
 Phase space $\RR\times(\RR\setminus\{0\})$ is also called in this case the time-frequency plane, since the first component determines the time, and the second component determines the frequency of the atom. The CWT is a Parseval frame if (\ref{eq:addmiss}) is equal to $1$. For a compactly supported mother wavelet, the time supports of the CWT atoms are inverse proportional to the frequency of the atoms.

\subsection{Signal processing in phase space}
\label{Signal processing in phase space}

In this paper we consider signal processing methods where first $V_f[s](g)$ is computed for every $g\in G$, followed by a nonlinearity $\k\Big(V_f[s](g),g\Big)$, where $\kappa:\CC\times G\rightarrow\CC$. Then, each $g\in G$ is mapped to $\rho(g)$ where $\rho:G\rightarrow G$ is integrable, and the resulting atoms are synthesized to the output signal
\[s_{\rm out} = \iint_G \k\Big(V_f[s](g),g\Big)f_{\rho(g)} dg.\]

All of the phase space signal processing examples in the introduction are of this form (multipliers, signal denoising e.g. wavelet shrinkage denoising, and phase vocoder with integer dilation).
In multipliers $\k(c,g)=cr(g)$ for every $c\in\CC$ and $g\in G$, where $r:G\rightarrow\CC$ is some function, and $\rho(g)=g$ \textcolor{black}{\cite{ex1,ex2,wave_mult0,New_mult0,New_mult1,New_mult2}}. In signal denoising $\k(c,g)=\k(c)$ depends only on $c$, and $\rho(g)=g$ \textcolor{black}{\cite{ex5,ex6,ex7}}. 

A \emph{time stretching phase vocoder} is an audio effect that slows down an audio signal without dilating its frequency content. In the classical definition, $G$ is the time frequency plane, and $V_f$ is the STFT. When the signal is dilated by an integer $D$, we consider the diffeomorphism operator
\[\rho(g_1,g_2)=(D g_1,g_2).\]
We consider the nonlinearity $\k(c,g)=\k(c)$, defined by $\k(e^{i\theta}a)=e^{iD\theta}a$, for $a,\theta\in\RR_+$. It is evident from this description that the signal is time-dilated by dilating the position of the time-frequency atoms, without dilating their frequency. The intensities of the atoms are retained, but their phases are modified so that the oscillations of neighboring atoms have compatible phases, avoiding destructive interference (See for example \cite{vocoder_book} or \cite{Ours1} for the explanation of the phase correction nonlinearity $\k$).

\subsection{The localizing time-frequency transform}

The localizing time-frequency transform (LTFT) is a combination of the STFT for low and high frequencies, and CWT atoms for middle frequencies. As a result, the time spread of the atoms is bounded from above by the time spread of the low frequency STFT window, and becomes shorter the higher the frequency. A third parameter in the phase space of the LTFT controls the number of oscillations in the mother wavelet. We consider real valued time signals $s$, and since such signals are uniquely determined by the positive side of the frequency domain, we may assume that $\hs$ is supported in $(0,\infty)$ without loss of generality.


\textcolor{black}{
Before defining the LTFT, in the following definition we formalize geometric characteristics of time-frequency atoms.
\begin{definition}
Let $q\in L^2(\RR)$. 
\begin{itemize}
    \item 
    The \emph{time-expected value} and the \emph{frequency-expected value} of $q$ are defined respectively as
\[e^{\rm T}_q = \int_{\RR}t\abs{q(t)}^2dt , \quad e^{\rm F}_q = \int_{\RR}\w\abs{\hat{q}(\w)}^2d\w,\]
whenever these integrals are finite.
 The function $q$ is said to be \emph{centered about $x$ in time} if $e^{\rm T}_q=x$, and \emph{centered about $\w$ in frequency} if $e^{\rm F}_q=\w$.
 \item
 If $q$ is supported on the interval $(t_1,t_2)$ and centered about $\kappa$ in  frequency,  the \emph{number of oscillations} in $\mathcal{M}(\w) q$ is defined to be $(\kappa+\abs{\w})(t_2-t_1)$.
\end{itemize}
\end{definition}}

The following definition is taken from \cite{Ours1,Ours2}, with a modified parameterization of the time-frequency space.

\begin{definition}[The localizing time-frequency continuous frame]
\label{def:LTFT}
The Localizing Time-Frequency Transform (LTFT) is based on the following parameters
\begin{itemize}
\item
\emph{A window function} $f\in L^2(\RR)$ supported on $(-\frac{1}{2},\frac{1}{2})$, and localized both in time and frequency about 0.
\item
The \emph{LTFT phase space} $G=\RR^2\times[0,1] = (time\times frequency\times oscillations)$ with the usual Lebesgue measure. 
\item
The \emph{LTFT-CWT transition frequencies} $0<b_0<b_1\in\RR$.
\item
The \emph{minimal number of wavelet oscillations} $\gamma$.
\item
The \emph{oscillation range} $\xi>0$.
\end{itemize}
The atoms of the LTFT are defined for every $(a,b,c)\in G$ by
\begin{equation}
f_{a,b,c}(x)  = [\tau(a,b,c)f](x)
= \left\{
\begin{array}{ccc}
	\sqrt{\frac{b_0}{\gamma}}e^{2i\pi(\frac{\xi}{\gamma} c b_0+ b)(x-a)}f(\frac{b_0}{\gamma}(x-a)) & {\rm if} & b<b_0 \\
	  \sqrt{\frac{b}{\gamma}}e^{2i\pi(\frac{\xi}{\gamma} c + 1)b(x-a)}f(\frac{b}{\gamma}(x-a)) & {\rm if} & b_0<b<b_1 \\
	\sqrt{\frac{b_1}{\gamma}}e^{2i\pi(\frac{\xi}{\gamma} c b_1+ b)(x-a)}f(\frac{b_1}{\gamma}(x-a))  & {\rm if} &  b>b_1
\end{array}
\right.
\label{eq:LTFT_atom}
\end{equation}
where the unitary operator $\tau(a,b,c):L^2(\RR)\rightarrow L^2(\RR)$ is defined for any $(a,b,c)\in G$ by (\ref{eq:LTFT_atom}).
\textcolor{black}{The \emph{maximal atom support} is defined as $S_0=\frac{\gamma}{b_0}$ and the \emph{minimal atom support} as $S_1=\frac{\gamma}{b_1}$.}
\end{definition}
\textcolor{black}{The LTFT atom $f_{a,b,c}$ is localized about the time $a$, and about the frequency $(\frac{\xi}{\gamma} c b_0+ b)$,  $(\frac{\xi}{\gamma} c + 1)$, or $(\frac{\xi}{\gamma} c b_1+ b)$, whenever $b<b_0$, $b_0<b<b_1$, or $b>b_1$,  respectively.}
We call the first and last case of (\ref{eq:LTFT_atom}) STFT atoms, and the second case CWT atoms. Note that the time support of each CWT atom $f_{a,b,c}$ is of length $\frac{\gamma}{b}$.
The number of oscillations in the CWT atom $f_{a,b,c}$ is $\gamma+\xi c$, independently of $b$.  The transform $\tau(a,b,c)$ of the LTFT are given explicitly  in the following lemma.
\begin{lemma}[Transforms of the LTFT]
\begin{equation}
 \tau(a,b,c)
= \left\{
\begin{array}{ccc}
	\mathcal{T}(a)\mathcal{M}(\frac{\xi}{\gamma}c b_0+b)\mathcal{D}(\frac{\gamma}{b_0}) & {\rm if} & b<b_0 \\
	  \mathcal{T}(a)\mathcal{M}\big((\frac{\xi}{\gamma}c +1)b\big)\mathcal{D}(\frac{\gamma}{b}) & {\rm if} & b_0<b<b_1 \\
\mathcal{T}(a)\mathcal{M}(\frac{\xi}{\gamma}c b_1+b)\mathcal{D}(\frac{\gamma}{b_1})  & {\rm if} &  b>b_1
\end{array}
\right.
\label{eq:LTFT_atomOP}
\end{equation}
\end{lemma}

\begin{example}[LTFT phase vocoder \cite{Ours1,Ours2}]
\label{LTFT phase vocoder cont}
The continuous LTFT phase vocoder is defined for $s\in L^2(\RR)$ by
\begin{equation}
s_{\rm out} = \iint_{G} \k\big(V_f[s](a,b,c)\big)f_{(D a,b,c)} dadbdc,
\label{eq:Dvocoder_cont}
\end{equation}
where $D$ is the dilation constant, and $\k(e^{i\theta}r)=e^{iD\theta}r$ for $\theta,r\in\RR_+$.
\end{example}


\section{Quasi-Monte Carlo signal processing in phase space}

In this section we describe the proposed Quasi-Monte Carlo signal processing in phase space.
We start by recalling the general Quasi-Monte Carlo method. We then introducing the general setting of quasi-Monte Carlo signal processing in phase space. We motivate the QMC method over standard methods in time-frequency analysis, and especially, in audio signal processing with phase vocoder. We also discuss the computational complexity of the method. Last, we derive an error analysis of QMC methods in general phase space transforms, and obtain corresponding error bounds for the QMC LTFT.

\subsection{Background: Quasi-Monte Carlo}
\label{General quasi Monte Carlo methods with low discrepancy point sets}

The material in this subsection is taken from \cite{dick_pillichshammer_2010,QMC0}.
Quasi-Monte Carlo (QMC) is a cubature method for approximating integrals.
Given a function $f:I^d\rightarrow\CC$, where $I^d=[0,1]^d$ and $d\in\NN$, a QMC is an approximations of the form
\begin{equation}
\int_{I^d} f(t)dt \approx \frac{1}{N}\sum_{n=1}^N f(x_n),
\label{eq:QMC_intro}
\end{equation}
where $\mathcal{P}_N=\{x_1,\ldots,x_N\}\subset I^d$ are sample points in $I^d$. The Koksma–Hlawka inequality  estimates the error in (\ref{eq:QMC_intro}) based on the star-discrepancy of $\mathcal{P}_N$ and the Hardy-Krause variation of $f$ \textcolor{black}{\cite[Theorem 5.1]{QMC0}}, as we recall next.

 Discrepancy describes the extent to which sample points can cover volumes. The \emph{star-discrepancy} of $\mathcal{P}_N$ is defined to be
\begin{equation}
D_N^*(\mathcal{P}_N)=\sup_{B\in {\rm Rec}^*}\abs{\frac{\#(B\cap \mathcal{P}_N)}{N}-\mu(B)}
\label{eq:star_disc}
\end{equation}
where $\mu$ is the Lebesgue measure of $\RR^d$, $\#(B\cap \mathcal{P}_N)$ is the number of points of $\mathcal{P}_N$ in $B$, and ${\rm Rec}^*$ is the set of rectangular boxes of the form
\[\prod_{j=1}^d\left.\left[0,u_j\right.\right)\]
with $0<u_j\leq 1$ \textcolor{black}{\cite[Formula (5.3)]{QMC0}}.

To define the Hardy-Krause variation we first recall multi-index notations. A \emph{multi-index} is a vector $\a=(\a_1,\ldots,\a_d)\in\NN_0^d$, where $\NN_0$ are the non-negative integers. For two multi-indices $\a,\b$, we write $\a\leq\b$ if $\a_j\leq\b_j$ for every $j=1,\ldots,d$. For a multi-index $\a$ we define the derivative $\partial_{\a}$ of functions of the variables $(a_1,\ldots,a_d)$ by
\[\partial_{\a}:= \prod_{j=1}^d \frac{\partial^{\a_j}}{\partial_{({a_j}^{\a_j})}}.\]
Let $\Lambda$ be the set of multi-indices $\a=(\a_1,\ldots,\a_d)$, with $\a_j\in\{0,1\}$ for every $j=1,\ldots,d$.
For every $\a\in\Lambda$ and integrable $g:I^d\rightarrow\CC$, denote by 
\begin{equation}
\iint_{I^d|_{\a}}g(a_1,\ldots,a_d)da^{\a}
\label{eq:Vint1}
\end{equation} 
the integration of $g$ with respect to all of the variables $a_j$ such that $\a_j=1$, and substitution of $1$ in all of the variables $a_j$ such that $\a_j=0$. For example
\[
\iint_{I^3|_{(1,0,1)}}g(a_1,a_2,a_3)da^{(1,0,1)} = \int_0^1\int_0^1 g(a_1,1,a_3) da_1 da_3.
\]
When $g$ is defined in a general rectangle $\prod_{j=1}^d[v_j,u_j]\subset \RR^d$, we substitute the greater edge points $u_j$ in the definition of (\ref{eq:Vint1}) instead of $1$.
The \emph{Hardy-Krause variation} of a smooth enough $f$ is defined to be \textcolor{black}{\cite[Formula (5.8)]{QMC0}}
\begin{equation}
\mathcal{V}(f)=\sum_{\a\in\Lambda} \iint_{I^d|_{\a}}\abs{\partial_{\a}f(a_1,\ldots,a_d)}da^{\a}.
\label{eq:Hardy-Krause}
\end{equation}

We can now bound the error in QMC with respect to the star discrepancy and Hardy-Krause variation.
\begin{theorem}[The Koksma–Hlawka inequality \textcolor{black}{\cite[Theorem 5.1]{QMC0}}]
\label{KoksmaHlawka}
For every sequence $\mathcal{P}_N=\{x_1,\ldots, x_N\}$ and smooth enough $f:I^d\rightarrow\CC$ 
\[\abs{\int_{I^d} f(x)dx - \frac{1}{N}\sum_{n=1}^Nf(x_n)} \leq \mathcal{V}(f) D^*_N(\mathcal{P}_N).\]
\end{theorem}

As evident from the Koksma–Hlawka inequality, one way to guarantee a low error in the QMC method is to construct a sampling set $\mathcal{P}_N$ with low star discrepancy. There are two types of constructions of such sample sets. 
The \emph{closed} type, or \emph{low discrepancy sequence}, considers an infinite sequence of points $\mathcal{P}=\{x_n\}_{n=1}^{\infty}$, and for each $N\in\NN$, $\mathcal{P}_N$ is defined to be the first $N$ points of $\mathcal{P}$, $\{x_n\}_{n=1}^{N}$. There are constructions of low discrepancy sequences with
\[D^*_N(\mathcal{P}_N) \leq C \frac{(\ln N)^d}{N},\]
where $C$ is some constant (for example Halton sequence \cite{HaltonS}).
The \emph{closed} type, or \emph{low discrepancy point set}, considers a sequence of sample sets $\mathcal{P}_N$, where $\mathcal{P}_N$ is a completely different set for each $N\in\NN$. There are constructions of low discrepancy point sets with
\[D^*_N(\mathcal{P}_N) \leq C \frac{(\ln N)^{d-1}}{N},\]
where $C$ is some constant (for example, Hammersley point set \cite{HammersleyS}). Low discrepancy point sets achieve lower QMC asymptotic error than low discrepancy sequences. However, low discrepancy sequences allow improving an existing approximation of an integral by adding new sample points to the current cubature sum, instead of computing a whole new cubature some when $N$ is increased.

\subsection{General setting of QMC time-frequency analysis}

In \cite{Ours1} the class of \emph{linear volume discretizable frames} was introduces, which can be explained in simple words as follows. Let $f$ be a continuous frame on the signal space $\cH$ and phase space $G$. Let $\cS \subset\cH$ be a class of signal. A discretization of $\cS $ is a sequence of subspaces $\cS_M\subset \cH$, $M\in\NN$, of dimension $M$ each. \textcolor{black}{We also call $M$ the \emph{resolution} of $\cS_M$.} \textcolor{black}{The spaces $\cS_M$ are chosen to satisfy the following approximation property: for every $s\in \cS $ and $\epsilon>0$, there exists $M_0\in\NN$ such that for every $M>M_0$ there exists $s_M\in \cS_M$ satisfying $\norm{s-s_M}<\epsilon$.} For some transforms, like the STFT \cite{Ours1}, CWT and LTFT \cite{Ours2}, most of the energy of $V_f[s_M]$ of signals $s_M\in \cS_M$ is concentrated about a domain $G_M\subset G$ of volume $O(M)$. Frames that have discretizations satisfying the above property are called \emph{linear volume discretizable}. 
\textcolor{black}{
\begin{definition}[Linear volume discretizable frame \cite{Ours1}]
\label{D:linear area discretizable}
Let $f:G\rightarrow\cH$ be a continuous frame.
Let $\cS \subset \cH$ be a class of signals, and $\{\cS_M\}_{m=1}^{\infty}$ a discretization of $\cS$. 
	The continuous frame $f$ is called \emph{linear volume discretizable} with respect to the class $\cS$ and the discretization $\{\cS_M\}_{M=1}^{\infty}$, if for every error tolerance $\epsilon>0$ there is a constant $C^{\epsilon}>0$ and $M_0\in\NN$, such that for any $M\geq M_0$ there is a subset $G_M\subset G$ with measure
\begin{equation}
\mu(G_M) \leq C^{\epsilon}{\rm dim}(\cS_M)
\label{eq:lin_area1}
\end{equation}
such that for any $s_M\in  \cS_M$,
\begin{equation}
\frac{\norm{V_f[s_M] - \chi_{G_M} V_f[s_M]}_2}{\norm{V_f[s_M]}_2} < \epsilon,
\label{eq:lin_area2}
\end{equation}
where $\chi_{G_M}$ is the characteristic function of the set $G_M$.
\end{definition}
}

Motivated by this discussion,  we consider a continuous frame $f$, restricted to some discrete signal space $\cS_M$ of dimension/resolution $M$, and restricted to the rectangular domain $G_M$ in phase space of measure $O(M)$.
A sampling discretization of a phase space transform $V_f$ is the restriction of the analysis operator $V_f[s]$ to a finite sample set $\mathcal{P}_N=(g_1,\ldots,g_N)\subset G_M$,
and approximation of the synthesis operator \textcolor{black}{$V_f^*:L^2(G)\rightarrow \cH$ by 
\begin{equation}
V_f^{M,N*}:\CC^{\mathcal{P}_N}\rightarrow \cH, \quad V^{M,N*}_f\big(F(g_n)\big)_n = \frac{\mu(G_M)}{N}\sum_{n=1}^N F(g_n)f_{g_n},
\label{eq:VfN_star}
\end{equation}
where  $\CC^{\mathcal{P}_N}$ is the set of functions ${\mathcal{P}_N}\rightarrow\CC$.} 

A quasi-Monte Carlo discretization of a phase space signal processing method is given by
\begin{equation}
s^N_{\rm out} = \frac{\mu(G_M)}{N}\sum_{n=1}^N \k\Big(V_f[s](g_n),g_n\Big)f_{\rho(g_n)}
\label{eq:QMC_processing}
\end{equation}
where $\k,\rho$ are defined in Subsection \ref{Signal processing in phase space}. We consider in this paper Euclidean rectangular phase spaces $G_M$. The sample set $\mathcal{P}_N$ is taken as the affine linear rescaling of a low discrepancy sequence/point-set in $[0,1]^d$, so it covers $G_M$.

\textcolor{black}{
In Appendix \ref{LTFT discretization} we explain how to discretize the signal space of the LTFT transform. We consider discrete signals represented by $M$ time samples, with a sample-rate of $L$ samples per time unit.  Under this discretization, the LTFT of any discrete signal $s_M$ has most of its energy localized in phase space about the compact domain
\[G_M= [-M/L-S_0,M/L+S_0]\times [0,L]\times [0,1],\]
with $S_0$ from Definition \ref{def:LTFT}.
We hence restrict the phase space of the LTFT to $G_M$, calling the restricted system LTFT$^M$. We denote the synthesis operator of LTFT$^M$ by $V_f^{* M}$, namely,
\[V_f^{* M}F= \iint_{G_M} F(a,b,c)f_{a,b,c} dadbdc.\]
}

\begin{example}
\label{LTFT phase vocoder disc}
By Example \ref{LTFT phase vocoder cont}, the QMC LTFT phase vocoder is defined by
\begin{equation}
s_{\rm out} = \frac{\mu(G_M)}{N}\sum_{n=1}^N \k\big(V_f[s](a_n,b_n,c_n)\big)f_{(D a_n,b_n,c_n)}
\label{eq:Dvocoder}
\end{equation}
where $a_n$ is the time value, $b_n$ the frequency value, and $c_n$ is the oscillation value. Here, $D$ is the dilation constant, and $\k(e^{i\theta}r)=e^{iD\theta}r$ for $\theta,r\in\RR$.
\end{example}

\subsection{Motivation for quasi-Monte Carlo discretization of the time-frequency space}

Generally, we have three requirements for the discrete atom system $\{f_{g_n}\}_{n=1}^N$ in time-frequency signal processing \textcolor{black}{(and also general signal processing in phase space).}
\begin{enumerate}
	\item 
	The time-frequency samples should be well-spread in the time-frequency plane, so we can treat the analysis transform as a feature extraction method of local frequencies.
	\item
	The sampling should allow reconstruction up to some small error.
	\item
	The sample set should be small enough to be computationally efficient.
\end{enumerate}

In the following we focus on time-frequency analysis, 
with $M$ the resolution of the discrete signal space. 
We consider time-frequency analysis based on CWT, STFT and LTFT, and compare the QMC discretization with DWT and discrete STFT. 
The second and third requirements are met by the QMC method, as is proved in Section  \ref{Computational complexity of Quasi-Monte Carlo LTFT} and \ref{Analysis of QMC phase space signal processing} for a class of continuous frames that contains the STFT, CWT, and LTFT. The second and third requirements are also met by standard discrete time-frequency methods, like discrete STFT, based on samples on a regular grid, and DWT, based on sample on a wavelet grid, e.g., dyadic samples (see Appendix \ref{A DWT construction}). The third requirement is not satisfied for regular discretizations of the LTFT, since the 3D time-frequency space requires more samples in comparison to 2D time-frequency methods.

To address the first requirement, we need to formalize the notion of well-spread samples. 
For that we consider the discrepancy of the sample set. \textcolor{black}{Suppose that the domain of interest in the time-frequency plane is a rectangle $G_M=[a_1,a_2]\times[b_1,b_2]$ of area $M$, where $M$ is also the resolution of the discrete signal space $\cS_M$. We scale this rectangle to $[0,1]^2$  using the mapping $\psi(a,b)=\big((a_2-a_1)a + a_1,(b_2-b_1)b + b_1\big)$. Namely, the scaling of the frame $f$ is defined as the fame 
\[\tilde{f}_{(a,b)} = \sqrt{M} f_{\psi(a,b)},\]
having the same frame bounds as $f$. Let $\mathcal{P}_N=(g_n)_{n=1}^N\subset G_M$ be the sample set in $G_M$, and $\tilde{\mathcal{P}}_N=(\tilde{g}_n=\psi^{-1}(g_n))_{n=1}^N\subset [0,1]$ the scaled sample set.}
The discrepancy of $\tilde{\mathcal{P}}_N$ is defined to be \textcolor{black}{\cite[Formula (5.2)]{QMC0}}
\begin{equation}
D_N(\tilde{\mathcal{P}}_N)=\sup_{B\in {\rm Rec}}\abs{\frac{\#(B\cap \tilde{\mathcal{P}}_N)}{N}-\mu(B)}
\label{eq:disc}
\end{equation}
with notations as in (\ref{eq:star_disc}), where ${\rm Rec}$ is the set of rectangles of the form
\[\left.\left[u_1,u_2\right.\right)\times \left.\left[v_1,v_2\right.\right) \subset [0,1]^2.\]
The discrepancy (\ref{eq:disc}) is comparable to the star discrepancy (\ref{eq:star_disc}) via
\[D^*_N(\tilde{\mathcal{P}}_N)\leq D_N(\tilde{\mathcal{P}}_N)\leq 2^d D^*_N(\tilde{\mathcal{P}}_N),\]
where in our case of a two-dimensional phase space $d=2$ \cite{dick_pillichshammer_2010}.

The discrepancy is a measure of uniformity, or spread, of sample points. For example, if there is square $Q\in {\rm Rec}$ that does not intersect the sample points, then $D_N(\mathcal{P}_N) \geq \mu(Q)$. The greater the area of the square $Q$, the higher this lower bound of the discrepancy.
In Appendix \ref{Discrepancy of DWT grids} we prove that the discrepancy of the wavelet grid is bounded from below by $\frac{C}{\sqrt{N}}$ for some constant $C$. This is also the discrepancy of the regular grid of the discrete STFT (a well-known result for regular grids \cite{dick_pillichshammer_2010}). Since there are sample sets with discrepancy $C'\frac{\log(N)}{N}$ (for example the Hammersley point set), which is asymptotically lower than $\frac{C}{\sqrt{N}}$, the DWT  and regular grids are not well spread, while the low discrepancy sample set of the QMC is optimally spread in the sense of discrepancy.
When adding the third axis of the LTFT, the discrepancies of the 3D DWT and discrete STFT grids are bounded from below by $\frac{C}{N^{\frac{1}{3}}}$. For comparison, there are QMC 3D sample sets with discrepancy $C'\frac{\log^2(N)}{N}$ (e.g., the Hammersley point set), which is asymptotically lower than $\frac{C}{N^{\frac{1}{3}}}$. In this sense QMC sample sets are better spread in the time-frequency plane than standard discretization methods.

\subsection{Time-frequency coverage with low discrepancy sample sets}

Let us describe a related point of view on the QMC samples, namely, the ``capacity'' of the samples to cover the time-frequency plane.  In principle, no atom can perfectly represent a unique frequency at a unique time. Instead, when measuring time-frequency coefficients via an analysis transform, each atom at time-frequency $(a,b)$ has a large interaction with a domain of time-frequency points in phase space about $(a,b)$. This is the essential domain covered by the time-frequency kernel $V_f(f_{a,b})$ centered at $(a,b)$ (also called the ambiguity function). We thus think of each atom as representing a small domain of times and frequency about $(a,b)$. The Heisenberg uncertainty principle informally states that the area of this domain is never less than some global positive constant, that can be assumed to be $1$ by choosing appropriate units of measurement. One way to represent this domain is by a \emph{Heisenberg box}, which is a rectangle of area 1 centered at $(a,b)$, with sides parallel to the axes. Here, the side along the time direction represents the time spread of the atom, and the side along the frequency direction represents the frequency spread of the atom.  One way to characterize any sample grid of a discrete STFT, is as a grid for which the corresponding Heisenberg boxes tessellate phase space. Namely, if the time spread of an atom is less than its frequency spread, then the grid spacing along time should be smaller than the spacing along frequency \cite{Time_freq}.

In Appendix \ref{Time-frequency tessellation in CWT and LTFT analysis} we present an analogous notion for wavelet analysis. For wavelets, the domain in phase space covered by the time-frequency kernels $V_f(f_{a,b})$ has a funnel shape. Indeed,  since atoms of lower frequencies have higher time supports, the spread of the kernel in the time direction increases the lower the frequency. We thus call this domain a \emph{wavelet funnel} (see Figure \ref{fig:funnel}). We show in Appendix \ref{Time-frequency tessellation in CWT and LTFT analysis} that the wavelet funnels about a low discrepancy sample set cover phase space approximately uniformly. This means that all time-frequency pairs are roughly evenly represented by the sampled atoms in the QMC method.
Note that the wavelet funnels about the DWT grid also cover phase space. However, for the wavelet funnels about a regular grid to cover phase space, the grid must be of $O(M^2)$ samples, which is not practical. 
Similarly to the wavelet funnels, the 3D equivalent shape for LTFT also admits an approximate uniform cover of 3D phase space via the low discrepancy sample set of QMC (see Appendix \ref{Time-frequency tessellation in CWT and LTFT analysis}).

\begin{figure}[!ht]
\centering
\includegraphics[width=0.25\linewidth]{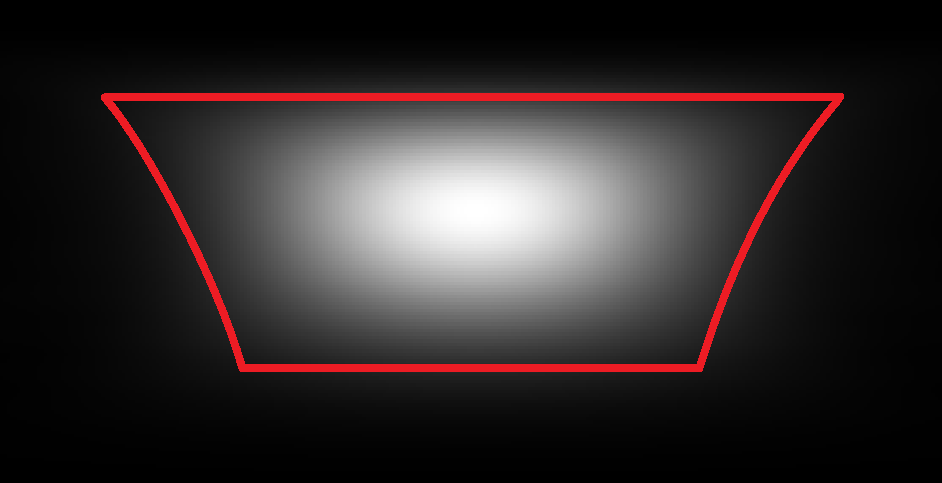}
\caption{The wavelet funnel represents the essential area covered by the wavelet kernel $V_f[f_{a,b}]$.
}
\label{fig:funnel}
\end{figure}

\subsection{Computational complexity of digital quasi-Monte Carlo LTFT}
\label{Computational complexity of Quasi-Monte Carlo LTFT}

In this subsection we estimate the computational complexity of the \textcolor{black}{analysis QMC LTFT transform $\big(V_f[s](g_n)\big)_{n=1}^N$ and synthesis QMC LTFT transforms $V_f^{M,N*}(F)$ of (\ref{eq:VfN_star}).}
\textcolor{black}{We first digitize the QMC LTFT. Let $L>0$ and $M\in\NN$ even. We consider a digital signal processing setting in which time signals are represented by $M$ equidistant sample points in the time interval $[-\frac{M}{2L},\frac{M}{2L}]$. Namely, the digital signal $\mathbf{x}=(x_m)_{m=-M/1}^{M/2}$ represents values at the points $(m/L)_{m=-M/2}^{M/2}$. The inner product between signals is estimated by
\begin{equation}
\label{eq:digitalIP}
\ip{\mathbf{x}}{\mathbf{y}} = \frac{L}{M}\sum_{m=-M/2}^{M/2} x_m\overline{y_m}.
\end{equation}  The Fourier transform is estimated by the discrete Fourier transform, where the $M$ discrete Fourier coefficients of digital signals are seen as points in the frequency interval $[0,L]$. Namely, $\mathbf{X}=(X_k)_{k=0}^M$ represents values at $(k/L)_{k=0}^M$, and
\begin{equation}
    \label{eq:DFT}
    X_k = \frac{L}{M}\sum_{m=-M/2}^{M/2}x_m e^{-2\pi i \frac{km}{M}}.
\end{equation}
When computing the analysis operator of the LTFT (Definition \ref{def:LTFT}), the continuous atoms $f_{a,b,c}$ are digitized by sampling them at $(m/L)_{m=-M/2}^{M/2}$,  and the inner product with  $\mathbf{x}$ is defined via (\ref{eq:digitalIP}).
}

\textcolor{black}{Motivated by (\ref{eq:DFT}), we consider the phase space $G_M=[-\frac{M}{2L},\frac{M}{2L}]\times[0,L]\times [0,1]$. Let $\mathcal{P}_N=\big((a_n,b_n,c_n)\big)_{n=1}^N$ be the $N$  quasi-Monte Carlo samples in $G_M$.}
 We take the transition frequencies of the LTFT (Definition \ref{def:LTFT}) as $b_0=C_1L$  and $b_1=C_2L$, with any choice of constants $C_1<C_2$ that do not depend on $N,M,L$. 

The size of the time support of the atom $f_{a,b,c}$ is 
\begin{equation}
S(a,b,c)=\left\{
\begin{array}{ccc}
	\gamma/b_0 & {\rm if} & 0\leq b \leq b_0\\
	\gamma/b & {\rm if} & b_0<b<b_1\\
	\gamma/b_1 & {\rm if} & b_1\leq b \leq L .
\end{array}
\right.
\label{eq:support4}
\end{equation}
%
%
The number of time samples in $f_{a,b,c}$ is estimated by $LS(b)$. \textcolor{black}{In both the analysis and synthesis LTFT transforms, the computational complexity $\mathcal{C}$ entailed by each sampled atom is proportional to the number of time samples in the atom. Hence, our goal is to count the overall number of samples in all atoms in the method, namely,
%
\begin{equation}
\mathcal{C}=\sum_{n=1}^N LS\Big(\frac{M}{L}\a_n,L \b_n,\c_n\Big).
\label{eq:comp1}
\end{equation}
where $(a_n,b_n,c_n)=(\frac{M}{L}\a_n,L\b_n,\c_n)$. Note that  $\big((\a_n,\b_n,\c_n)\big)_{n=1}^N$ is a low discrepancy sample set in $[0,1]^3$.}

To estimate (\ref{eq:comp1}) we use the Koksma-Hlawka inequality \textcolor{black}{(Theorem \ref{KoksmaHlawka})}  ``backwards,'' approximating \textcolor{black}{the sum (\ref{eq:comp1})  by the integral of $S(\a,\b,\c)$  (\ref{eq:support4}). While computing directly the sum  (\ref{eq:comp1}) is difficult, computing the integral of $S(\a,\b,\c)$ is easy.} 
Hence, by the Koksma-Hlawka inequality, \textcolor{black}{the complexity} $\mathcal{C}$ is estimated by
\begin{equation}
A(N,L):=N\iint_{[0,1]^3} LS\Big(\frac{M}{L}\a,L \b,\c\Big)d\a d\b d\c.
\label{eq:av1}
\end{equation}
up to an error or order
\[E(N,L):=L\mathcal{V}\Big(S\big(\frac{M}{L}\a,L \b,\c\big)\Big)\log^2(N),\]
where $\mathcal{V}(\cdot)$ is the Hardy-Krause variation \textcolor{black}{(\ref{eq:Hardy-Krause})}.
It is easy to see that
\[A(N,L) = \gamma N + \gamma \ln\Big(\frac{b_1}{b_0}\Big)N + \gamma \frac{L-b_1}{b_1}N.\]
For $b_0=C_1L$ and $b_1=C_2L$ we have
\[A(N,L) = O(\gamma N).\]
Moreover, \textcolor{black}{by (\ref{eq:support4}) $S$ depends only on $b$, so by the definition of the Hardy-Krause  variation (\ref{eq:Hardy-Krause})
\[
\begin{split}
  \mathcal{V}\Big(S\big(\frac{M}{L}\a,L \b,\c\big)\Big) & =  \abs{S\big(\frac{M}{L},L,1\big)}
+
\int_0^L\abs{\partial_{\b}(S\big(\frac{M}{L},L \b,1\big)}d\b\\
& = \frac{\gamma}{b_1}
+
\int_{b}\abs{ \partial_b S(\frac{M}{L},b,1)}db = \frac{\gamma}{b_1}
+
\int_{b_0}^{b_1} \frac{\gamma}{b^2}db =\frac{\gamma}{b_0}
\end{split}
\]
}
%
%
%

Thus, the computational complexity of the QMC LTFT method is
\[\mathcal{C} = O\big(\gamma (N + \log^2(N))\big) = O(N).\]

In Section \ref{Analysis of QMC phase space signal processing} we show that the QMC LTFT method has approximation error $O(\frac{M\log(N)^2}{N})$. Thus,
 if we choose
\[N= A M \log(M)^{2+\e},\]
 the error is
\[\frac{M\log(A M \log(M)^{2+\e})^2}{A M \log(M)^{2+\e}}<\frac{1}{A},\]
and vanishes asymptotically.


In practice, taking $N=A M$ for relatively small $A$ works well. For example, in integer time dilation phase vocoder with dilation constant $D$, taking  $A=4D$ gives high-quality results. Any value of $A$ greater than $4 D$ does not improve the audible quality of the method. Sound examples and code of QMC LTFT  phase vocoder are available at \url{https://github.com/RonLevie/LTFT-Phase-Vocoder}.

\section{Approximation analysis of QMC phase space signal processing}
\label{Analysis of QMC phase space signal processing}

In this section we present a class of continuous frame transforms that includes the STFT, CWT, LTFT, and systems like the Shearlet \cite{Shearlet} and the Curvelet \cite{Curvelet} transforms. We then analyze the QMC approximation of these transforms, proving that the error rate is of the form $O(\frac{\log(N)^{d-1}M}{N})$ for $N$ QMC samples, discrete signals of resolution $M$, and phase space of dimension $d$. Last, we compute the error rate in the example of the LTFT, and particularly for QMC phase vocoder.

\subsection{General phase space transforms}

We define a setting which generalizes the discretization of the LTFT  of Appendix \ref{LTFT discretization}.  
For $A=a^{c}$, we denote by $\mathcal{R}^{c}(A)$ the square in $\RR^{c}$, centered at $0$ with sides $a>0$, \textcolor{black}{namely, $\mathcal{R}^{c}(A)$ is the square of volume $A$.} We denote by $\mathcal{T}(\mathbf{a})$ the translation by $\mathbf{a}\in\RR^d$ in $L^2(\RR^d)$ defined by $[\mathcal{T}(\mathbf{a})f](x)=f(x-\mathbf{a})$.

\begin{assumption}[Discrete signal space and transform setting]
\label{As2}
$ $
\begin{enumerate}
\item
The continuous signal space is $\cS = L_{2}(\RR^{d_s})\cap L_{\infty}(\RR^{d_s})$, where $d_s\in\NN$.  
	 \item 
	 Let $L\geq 1$ be a constant that we call the \emph{formal frequency support}.
	 For each $m\in\NN$, there is an $M=m^{d_s}$ dimensional discrete space $\cS_M\subset L_{\infty}(\RR^{d_s})\cap L_2(\RR^{d_s})$ of signals supported in $\mathcal{R}^{d_s}(M/L)$.
	 \item
	\label{Delta}
The atom system is given by 
	\[\{f_{(\mathbf{a},\mathbf{b})}:=\mathcal{T}({\bf a})f_{\bf b}\}_{({\bf a}, {\bf b})\in \RR^{d}},\]
	where for each ${\bf b}$, $f_{\bf b}$ is supported on a square of sides less than the global constant $\Delta>0$. Here, ${\bf a}\in\RR^{d_s}$ is the \emph{position parameter} and ${\bf b}\in\RR^{d-d_s}$ is the \emph{formal frequency parameter}.
\item
Define $l=L^{\frac{1}{d_s}}$. For each $m\in\NN$ define $m'=m+\left\lceil \Delta l \right\rceil$ and $M'=(m')^{d_s}$. The compact phase space of dimension $d$ is defined to be
\[G_{M}= \mathcal{R}^{d_s}(M'/L)\times \mathcal{R}^{d-d_s}(L),\]
with the standard Lebesgue measure. 
	\item
	For every $x\in \RR^{d_s}$ the mapping $({\bf a},{\bf b})\mapsto \mathcal{T}({\bf a})f_{\bf b}(x)$ is assumed to be differentiable with respect to any differential operator $\partial_{\a}$, with $\a=(\a_1,\ldots,\a_d)\in \Lambda$ (see Subsection \ref{General quasi Monte Carlo methods with low discrepancy point sets} for $\Lambda$), and for every $\mathbf{a}\in\RR^{d_s}$
	\begin{equation}
	\norm{\partial_{\a}[\pi({\bf a})f_{\bf b}]}_1 \leq H(\a;{\bf b}).
	\label{eq:pi_bound}
	\end{equation}
	Here, $\{H(\a;\cdot)\}_{\a\in\Lambda}$ is a set of integrable functions that depend only on ${\bf b}$ for each $\alpha$, and the integration in the $L_1(\RR^{d_s})$ norm in (\ref{eq:pi_bound}) is with respect to the signal domain variable $x$.
	\item
	 Analysis $V^M_f[s]({\bf a},{\bf b})$ and synthesis $V^{M*}_f[S]$ are defined for $({\bf a},{\bf b})\in G_{M}$ by (\ref{eq:CSS2frame}) and (\ref{eq:inver_proj}), \textcolor{black}{namely,
	 \begin{equation}
	 V^M_f[s]:G_{M}\rightarrow \CC \quad , \quad V^M_f[s](\mathbf{a},\mathbf{b})=\ip{s}{f_{(\mathbf{a},\mathbf{b})}}_{\cH}.
	     \label{eq:M_An}
	 \end{equation}
	 \begin{equation}
	  V^{M*}_f[S] = \int_{G_{M}} S(\mathbf{a},\mathbf{b})f_{(\mathbf{a},\mathbf{b})} d\mathbf{a}d\mathbf{b}.
	     \label{eq:M_Sn}
	 \end{equation}}
\end{enumerate}
\end{assumption}

\textcolor{black}{
\begin{remark}
The domain $G_M$ in phase space of Assumption \ref{As2}.4 is a direct product of the \emph{spatial square}  $\mathcal{R}^{d_s}( M' /L)$ and the \emph{frequency square} $\mathcal{R}^{d-d_s}(L)$. The spatial square $\mathcal{R}^{d_s}( M' /L)$ is derived from the signal domain $\mathcal{R}^{d_s}(M/L)$ by increasing each side of $\mathcal{R}^{d_s}(M/L)$ by $\left\lceil \Delta \right\rceil$. Note that $\left\lceil \Delta \right\rceil$ bounds the support of atoms $f_{(\mathbf{a},\mathbf{b})}$.
By definition, any discrete signal $s_M\in \cS_M$ is supported in $\mathcal{R}^{d_s}(M/L)$. Hence, for any $\mathbf{a}$ in the boundary of $\mathcal{R}^{d_s}( M'/L)$, and any $\mathbf{b}\in \mathcal{R}^{d-d_s}(L)$, the signal $s_M$ and the atom $f_{(\mathbf{a},\mathbf{b})}$ have non-intersecting supports, so 
\[V_f^M[s_M](\mathbf{a},\mathbf{b})
=\ip{s_M}{f_{(\mathbf{a},\mathbf{b})}}=0.\]
\end{remark} 
}

\textcolor{black}{
In Subsection \ref{Quasi-Monte Carlo LTFT error} we present a setting that satisfies Assumption \ref{As2} for the LTFT.}




\subsection{Approximation rate of general QMC phase space transforms}

For a sample set $\mathcal{P}_N=\{(\mathbf{a}^1,\mathbf{b}^1),\ldots,(\mathbf{a}^N,\mathbf{b}^N)\}\subset G_{M}$, the QMC approximation \textcolor{black}{(\ref{eq:VfN_star})} of $V^{M*}_f(S)$ takes the form
\begin{equation}
V^{M,N*}_f(S)=\frac{M'}{N}\sum_{n=1}^N S(\mathbf{a}^n,\mathbf{b}^n)\pi(\mathbf{a}^n)f_{\mathbf{b}^n},
\label{eq:QMC_Vf}
\end{equation}
where $S:G'\rightarrow\CC$.
Let $\Gamma\subset\Lambda$ be the set of multi-indices with $\a_j=1$ for every spatial index $j=1,\ldots,d_s$ (see Subsection \ref{General quasi Monte Carlo methods with low discrepancy point sets} for $\Lambda$). In the following we bound the quasi-Monte Carlo synthesis error uniformly pointwise. 

\begin{theorem}
\label{T1}
Consider the setting of Assumption \ref{As2}. 
Consider a low discrepancy sample set $\mathcal{P}_N=\{(\mathbf{a}^1,\mathbf{b}^1),\ldots,(\mathbf{a}^N,\mathbf{b}^N)\}\subset G_{M}$ with
\begin{equation}
D^*_N(\mathcal{P}_N)  \leq  C\frac{\log(N)^{d-1}M}{N}.
\label{eq:LDSS}
\end{equation}
and a discrete signal $s_M\in \cS_M$.
 Then,
\begin{equation}
\frac{\norm{V^{M*}_fV^M_f(s_M) - V^{M,N *}_fV^M_f(s_M)}_{\infty}}{\norm{s_M}_{\infty}}\leq C\frac{\log(N)^{d-1}M(1+\e_M)}{N}D
\label{eq:T1}
\end{equation}
where 
\begin{equation}
\e_M=(1+\frac{\left\lceil \Delta l \right\rceil}{M^{1/d_s}})^{d_s}-1 \xrightarrow[M \to 0]{} 0,
\label{eq:epsM}
\end{equation}
and
\[D=D_M=\sum_{\a\in \Gamma}\sum_{\b\in\NN_0^d:\b\leq \a} 
\left(
\begin{array}{c}
	\a \\
	\b
\end{array}
\right)
\iint_{\mathcal{R}^{d-d_s}(L)|_{[\a]_{\mathbf{b}}}}  H(\b;\mathbf{b}) H(\a-\b;\mathbf{b}) \ d\mathbf{b}^{[\a]_{\mathbf{b}}}.\]
Here, $[\a]_{\mathbf{b}}$ is the restriction of the multi-index $\alpha$ to the frequency coordinates $\mathbf{b}$, and $\mathcal{R}^{d-d_s}(L)$ is the domain of the frequency coordinates $\mathbf{b}$ in $G_{M}$.
Moreover,
 for every multi-index $\a\in\Lambda$ we have
\begin{equation}
\abs{\partial_{\a}V^M_f[s_M](\mathbf{a},\mathbf{b})}\leq \norm{s_M}_{\infty} H(\a;\mathbf{b}).
\label{eq:bound_Vfs}
\end{equation}

\end{theorem}

In the proof of Theorem \ref{T1}, the bound (\ref{eq:T1}) is derived from (\ref{eq:pi_bound}) and (\ref{eq:bound_Vfs}). In a signal processing methods that transform $V_f[s_M]$ to a function $S$ that preserves the bound (\ref{eq:bound_Vfs}), namely,
\[\abs{\partial_{\a}S(\mathbf{a},\mathbf{b})}\leq C' H(\a;\mathbf{b}),\]
the error estimate (\ref{eq:T1}) still holds with the constant $C'$ instead of $\norm{s_M}_{\infty}$. More generally, we consider signal processing tasks that transform $V_f[s_M]$ to functions $S$ subject to some generic bound. The following theorem estimates the QMC error of synthesizing in this case.

\begin{theorem}
\label{T2}
Consider the setting of Assumption \ref{As2}.
Consider a low discrepancy sample set $\mathcal{P}_N=\{(\mathbf{a}^1,\mathbf{b}^1),\ldots,(\mathbf{a}^N,\mathbf{b}^N)\}\subset G_{M}$ satisfying (\ref{eq:LDSS}).  
Consider  a phase space function $S\in L_2(G_{M})$ satisfying the following two conditions.
\begin{enumerate}
	\item 
	Boundedness: for every multi-index $\a\in\Lambda$
	\[\abs{\partial_{\a}S(\mathbf{a},\mathbf{b})} \leq F(\a;\mathbf{b}),\]
	where $\{F(\a;\cdot)\}_{\a}$ is a set of integrable functions that depend only on $\mathbf{b}$ for each $\alpha$. 
	\item
	\label{T2.2}
	Vanishing spatial boundary condition:  for every multi-index $\a\in\Lambda$, $\partial_{\a}S(\mathbf{a},\mathbf{b})=0$ for every $\mathbf{b}$ and every $\mathbf{a}=(a_1,\ldots,a_{d_s})$ with at least one coordinate $a_j=m'/l$.
\end{enumerate}
 Then 
\begin{equation}
\norm{V^{M*}_f(S) - V^{M,N *}_f(S)}_{\infty}\leq C\frac{\log(N)^{d-1}M(1+\e_M)}{N}D,
\label{eq:T12}
\end{equation}
where $\e_M$ is given in (\ref{eq:epsM}), and
\begin{equation}
D=D_M=\sum_{\a\in\Gamma}\sum_{\b\in\NN_0^d:\b\leq \a} 
\left(
\begin{array}{c}
	\a \\
	\b
\end{array}
\right)
\iint_{\mathcal{R}^{d-d_s}(L)|_{[\a]_{\mathbf{b}}}}  H(\b;\mathbf{b}) F(\a-\b;\mathbf{b}) \ d\mathbf{b}^{[\a]_{\mathbf{b}}}.
\label{eq:DDM}
\end{equation}
\end{theorem}

\begin{remark}
\label{remark_D}
In general, the value $D$ in
Theorems \ref{T1} and \ref{T2} may depend on $M$. Thus, these theorems are only useful in situations where the constant $D$ is independent of $M$, or at least ``substantially sublinear'' in $M$. In Subsection \ref{Quasi-Monte Carlo LTFT error} we show that $D$ is independent of $M$ in QMC LTFT.
\end{remark}

We start by proving Theorem \ref{T2}.

\begin{proof}[Proof of Theorems \ref{T2}]
\textcolor{black}{We study in this proof the pointwise error between
\begin{equation}
   V^{M*}_f(S)(x)= \int_{G_{M}} S(\mathbf{a},\mathbf{b})\mathcal{T}(\mathbf{a})f_{\mathbf{b}}(x)d\mathbf{a}d\mathbf{b}
    \label{eq:T2_PW}
\end{equation}
and the QMC approximation  $V^{M,N *}_f(S)(x)$ of (\ref{eq:QMC_Vf}).}
As a result of the vanishing spatial boundary condition (Condition \ref{T2.2} of Theorem \ref{T2}), in the following analysis
we need to consider only multi-indices $\a\in\Gamma$ in the Hardy-Krause variation (\ref{eq:Hardy-Krause}).
Let $l=L^{\frac{1}{d_s}}$ and $j=L^{\frac{1}{d-d_s}}$. \textcolor{black}{Recall from Assumption \ref{As2} that for each $m\in\NN$, $m'=m+\left\lceil \Delta l \right\rceil$ and $M'=(m')^{d_s}$.}
To use the Koksma-Hlawka inequality (Theorem \ref{KoksmaHlawka}), we scale $G_{M}$ by a linear change of variables to $[0,1]^d$, and scale the integrand \textcolor{black}{of (\ref{eq:T2_PW})}
\[Q:G_{M}\rightarrow \CC , \quad  Q(\mathbf{a},\mathbf{b};x):= S(\mathbf{a},\mathbf{b})\mathcal{T}(\mathbf{a})f_{\mathbf{b}}(x)\]
to the integrand
\[W:[0,1]^d\rightarrow \CC , \quad  W(\mathbf{y},\mathbf{z};x):=M' S(\frac{m'}{l}\mathbf{y},j\mathbf{z})\mathcal{T}(\frac{m'}{l}\mathbf{y})f_{j\mathbf{z}}(x).\]
The Hardy-Krause variation of $W$ consists of terms, corresponding to multi-indexes $\a\in\Gamma$, of the form
\[M'\iint_{[0,1]^d|_{\a}} (\frac{m}{l},j)^{\a}\abs{[\partial_{\a}Q](\frac{m'}{l} \mathbf{y}, j \mathbf{z}; x)}{(d\mathbf{x},d\mathbf{y})}^{\a} = \iint_{G_{M}|_{\a}} (\frac{m'}{l},j)^{\a}\abs{\partial_{\a}Q( \mathbf{a},  \mathbf{b}; x)}{(d\mathbf{a},d\mathbf{b})}^{\a},\]
where 
\[(\frac{m'}{l},j)^{\a} := \prod_{k=1}^{d_s} (\frac{m'}{l})^{\a_k}\prod_{k=d_s+1}^{d} j^{\a_k}.\]
Thus, since $L\geq 1$, 
we have $(\frac{m'}{l},j)^{\a}\leq M'$, so
\[M'\iint_{[0,1]^d|_{\a}} (\frac{m'}{l},j)^{\a}\abs{[\partial_{\a}Q](\frac{m'}{l} \mathbf{x}, j \mathbf{y}; x)}{(d\mathbf{x},d\mathbf{y})}^{\a} \leq M'\iint_{G_{M}|_{\a}} \abs{\partial_{\a}Q( \mathbf{a}, \mathbf{b}; x)}{(d\mathbf{a},d\mathbf{b})}^{\a}.\]

Now, the theorem follows from the product rule as follows. We have
\[\begin{split} 
& \iint_{G_{M}|_{\a}} \abs{\partial_{\a}Q( \mathbf{a},  \mathbf{b}; x)}{(d\mathbf{a},d\mathbf{b})}^{\a} \\
 & \leq \sum_{\b\in\NN_0:\b\leq \a} 
\left(
\begin{array}{c}
	\a \\
	\b
\end{array}
\right)
\iint_{G_{M}|_{\a}}  \abs{\partial_{\b}\mathcal{T}(\mathbf{a})f_{\mathbf{b}}(x)} \abs{\partial_{\a-\b} S(\mathbf{a},\mathbf{b})} {(d\mathbf{a},d\mathbf{b})}^{\a}.
\end{split}\]
Consider the multi-index $\gamma = (\a_{d_s+1},\ldots,\a_d)$.
By the H\"older's inequality along the spatial direction $\mathbf{a}$ for each fixed $\mathbf{b}$, we have
\[\iint_{G_{M}|_{\a}}\abs{\partial_{\b}\mathcal{T}(\mathbf{a})f_{\mathbf{b}}(x)} \abs{\partial_{\a-\b} S(\mathbf{a},\mathbf{b})} {(d\mathbf{a},d\mathbf{b})}^{\a} \leq    \iint_{\mathcal{R}^{d-d_s}(L)|_{\gamma}} H(\b;\mathbf{b})F(\a-\b;\mathbf{b}){d\mathbf{b}}^{\gamma}.\]
This, together with the fact that we need only consider $\alpha\in\Gamma$, gives 
\begin{equation}
\norm{V^{M*}_f(S) - V^{M,N *}_f(S)}_{\infty}\leq C\frac{\log(N)^{d-1}M'}{N}D,
\label{eq:T120}
\end{equation}
with $D$ given in (\ref{eq:DDM}).
Last, note that
\[M'= (m+\left\lceil \Delta l \right\rceil)^{d_s}  = M(1+\frac{\Delta l}{M^{1/d_s}})^{d_s},\]
which gives (\ref{eq:T12}).

\end{proof}

\begin{proof}[Proof of Theorems \ref{T1}]

We first show \textcolor{black}{(\ref{eq:bound_Vfs}), namely,} for every multi-index $\a=(\a_1,\ldots,\a_d)\in \Lambda$ we have
\[
\abs{\partial_{\a}V^M_f[s_M](\mathbf{a},\mathbf{b})}\leq \norm{s_M}_{\infty} H(\a;\mathbf{b}).
\]
Indeed, \textcolor{black}{by (\ref{eq:M_An})}, H\"older's inequality, \textcolor{black}{and Leibniz integral rule,}  
\[\begin{split} 
\abs{\partial_{\a}V^M_f[s_M](\mathbf{a},\mathbf{b})}  & \leq \iint_{\mathcal{R}^{d_s}(M/L)} \abs{s_M(x) \partial_{\a}\mathcal{T}(\mathbf{a})f_{\mathbf{b}}(x)}dx \\
 & \leq \norm{s_M}_{\infty}\norm{\partial_{\a}\mathcal{T}(\mathbf{a})f_{\mathbf{b}}}_1 = \norm{s_M}_{\infty} H(\a;\mathbf{b}).
\end{split}\]

Now, Theorem \ref{T1} follows Theorem \ref{T2} with $F(\a;\mathbf{b})=\norm{s_M}_{\infty}H(\a;\mathbf{b})$ and $S=V_f[s_M]$. 
Indeed, the vanishing spatial boundary condition (Condition \ref{T2.2} of Theorem \ref{T2}) is satisfied for $S=V_f[s_M]$ since the supports of $\partial_{\a}\mathcal{T}(\mathbf{a})f_{\mathbf{b}}$ and $s_M$ are disjoint for $\mathbf{a}$ with at least one $a_j=m'/l$, and thus
\begin{equation}
\partial_{\a}V^M_f[s_M](\mathbf{a},\mathbf{b})  =  \iint_{\mathcal{R}^{d_s}(M/L)} s_M(x)\overline{\partial_{\a}\mathcal{T}(\mathbf{a})f_{\mathbf{b}}(x)} dx=0.
\label{eq:QWGS}
\end{equation} 

\end{proof}

\begin{example}[QMC multiplier]
\label{ex:QMC_multiplier}
An example application of Theorem \ref{T2} in signal processing is multipliers. Suppose $V_f[s]$ is multiplied by a function $Q:G\rightarrow\CC$, before synthesized to an output signal. Here, if all derivatives of $Q$ with $\partial_{\alpha}\in\Lambda$ are bounded in $L_{\infty}(G)$, then, by H\"older's inequality, the bound $H$ of (\ref{eq:pi_bound}) is preserved under the application of the multiplier (up to a constant). In this case, if $D$ of Theorem \ref{T1} is independent of $N$, then the QMC multiplier method has error rate of $O(\frac{\log(N)^{d-1}M}{N})$.
\end{example}

\subsection{Error analysis of Quasi-Monte Carlo LTFT}
\label{Quasi-Monte Carlo LTFT error}

In this section we use Theorems \ref{T1} and \ref{T2} to analyze QMC LTFT. 
As explained in Remark \ref{remark_D}, the goal in this section is to show that the $D$ constant of Theorems  \ref{T1} and \ref{T2} is independent of the signal resolution $M$, which shows that QMC LTFT has error rate $O(\frac{M\log^2(N)}{N})$.

Let $f$ be a twice continuously differentiable compactly supported window function.
We consider real valued time signals, and thus it is enough to analyze the positive half frequency line, since the complete signal can be reconstructed from this information \textcolor{black}{due to the Hermitian symmetry \cite{Asignal}} . 
Define the operators $\mathcal{D}_{j,l}$, for $j,l=0,1,2$, by
\[\mathcal{D}_{l,j}q(x) = x^l q^{[j]}(x).\]
For reasons that will become clear soon, we consider the set of windows
\[f_{l,j} = \mathcal{D}_{l,j}f \quad , \quad {\rm for\ }l,j=0,1,2.\]
Let $C>0$ satisfy
\begin{equation}
\label{f_lj_C}
  \forall l,j=0,1,2, \quad \norm{f_{l,j}}_1 < C.  
\end{equation}

Consider the following discretization.
We consider the ``sample-rate'' $L$, \textcolor{black}{and the discrete signal space $\cS_M\subset L_{\infty}(\RR^{d_s})\cap L_2(\RR^{d_s})$ of signals supported in $\mathcal{R}^{d_s}(M/L)$.}  
We consider phase space
\[G_M = [-\frac{1}{2}\frac{M}{L}, \frac{1}{2}\frac{M}{L}] \times [0,L] \times [0,1].\]
Here, we take $G_M$ instead of $G_{M}$ since this does not affect the asymptotic analysis for large $M$.

\begin{proposition}
\label{ClaimLTFT0}
Consider the above construction, and a low discrepancy point set $\mathcal{P}_N$. \textcolor{black}{Then, the LTFT based on $f$ satisfies Assumption \ref{As2},} and the QMC LTFT method satisfies the error bound
\begin{equation}
\frac{\norm{V^{M *}_fV^M_f[s_M] - V^{M,N *}_fV^M_f[s_M]}_{\infty}}{\norm{s_M}_{\infty}}\leq C\frac{\log(N)^{d-1}M(1+\e_M)}{N}D,
\label{eq:T1LTFT}
\end{equation}
where $D$ is a constant that depends only on $C$ of (\ref{f_lj_C}), the minimal number of oscillations $\gamma$, and the oscillation range $\xi$ (see Definition \ref{def:LTFT}).
\end{proposition}

\begin{proof}
We compute the integrals (\ref{eq:DDM}) in the three subdomains $[0,b_0]$, $(b_0,b_1)$, and $[b_1,L]$, differentiating each case of (\ref{eq:LTFT_atom}) separately, which is enough by continuity.
Let us start with the subdomain of $G_{M}$ of CWT atoms, namely, the middle frequencies.
%
To compute the $H$ functions of (\ref{eq:pi_bound}) in Assumption \ref{As2},  we first compute the derivatives $\partial_{\a}f_{a,b,c}$. We have
\[\partial_a \tau(a,b,c)f(x) = \partial_a \Big( \sqrt{\frac{b}{\gamma}}e^{2i\pi(\frac{\xi}{\gamma} c + 1)b(x-a)}f\big(\frac{b}{\gamma}(x-a)\big)\Big)\]
\begin{equation}
\quad\quad =  -b\Big(2i\pi(\frac{\xi}{\gamma} c + 1)\tau(a,b,c)\mathcal{D}_{0,0}f(x)  +\frac{1}{\gamma} \tau(a,b,c)\mathcal{D}_{0,1}f(x)\Big).
\label{eq:WP1}
\end{equation}
Moreover,
\[\begin{split}\partial_b \tau(a,b,c)f(x) & = \partial_b \Big( \sqrt{\frac{b}{\gamma}}e^{2i\pi(\frac{\xi}{\gamma} c + 1)b(x-a)}f\big(\frac{b}{\gamma}(x-a)\big)\Big)\\
& =\frac{1}{2}\frac{1}{b}\sqrt{\frac{b}{\gamma}}e^{2i\pi(\frac{\xi}{\gamma} c + 1)b(x-a)}f\big(\frac{b}{\gamma}(x-a)\big)\\
& \quad + \frac{1}{b}  2i\pi(\xi c + \gamma)\sqrt{\frac{b}{\gamma}}e^{2i\pi(\frac{\xi}{\gamma} c + 1)b(x-a)}\frac{b}{\gamma}(x-a)f\big(\frac{b}{\gamma}(x-a)\big)\\
 & \quad +\frac{1}{b}\sqrt{\frac{b}{\gamma}}e^{2i\pi(\frac{\xi}{\gamma} c + 1)b(x-a)}\frac{b}{\gamma}(x-a)f'\big(\frac{b}{\gamma}(x-a)\big)
 \end{split}\] 
\begin{equation}
= \frac{1}{b}\Big(\frac{1}{2}\tau(a,b,c)\mathcal{D}_{0,0}f(x) + 2i\pi(  \xi c+\gamma )\tau(a,b,c)\mathcal{D}_{1,0}f(x) + \tau(a,b,c)\mathcal{D}_{1,1}f(x) \Big).
\label{eq:WP2}
\end{equation}
Last, similarly to the above calculation,
\[\partial_c \tau(a,b,c)f(x) = \partial_c \Big( \sqrt{\frac{b}{\gamma}}e^{2i\pi(\frac{\xi}{\gamma} c + 1)b(x-a)}f\big(\frac{b}{\gamma}(x-a)\big)\Big)\]
\begin{equation}
=2i\pi\xi\tau(a,b,c)\mathcal{D}_{1,0}f(x).
\label{eq:WP3}
\end{equation}

Next, we construct $H$ bounds satisfying (\ref{eq:pi_bound}).
We have
\[\norm{f_{a,b,c}}_1 = \int \sqrt{\frac{b}{\gamma}}\abs{f\big(\frac{b}{\gamma}(x-a)\big)} dx=b^{-0.5}\gamma^{0.5}\norm{f}_1,\]
so, by (\ref{f_lj_C}), we choose for $b_0<b<b_1$
\[H(I;b)=b^{-0.5}\gamma^{0.5}C.\]
\textcolor{black}{
Moreover, by  (\ref{f_lj_C}) and (\ref{eq:WP1}),
\[
\begin{split}
  \norm{\partial_a f_{a,b,c}}_1& = 
\int \abs{  b\Big(2i\pi(\frac{\xi}{\gamma} c + 1)\tau(a,b,c)\mathcal{D}_{0,0}f(x)  +\frac{1}{\gamma} \tau(a,b,c)\mathcal{D}_{0,1}f(x)\Big)} dx  \\
& \leq b2\pi(\frac{\xi}{\gamma} c + 1)\int \abs{  (\tau(a,b,c)\mathcal{D}_{0,0}f(x)} dx + \frac{b}{\gamma}\int\abs{  \tau(a,b,c)\mathcal{D}_{0,1}f(x)} dx\\
& \leq b2\pi(\frac{\xi}{\gamma} c + 1) b^{-0.5}\gamma^{0.5}\norm{\mathcal{D}_{0,0}f}_1 + \frac{b}{\gamma}b^{-0.5}\gamma^{0.5}\norm{\mathcal{D}_{0,1}f}_1 dx.
\end{split}
\]
Hence, we choose
\[H(\partial_a;b) = b^{0.5}\gamma^{0.5}C\big(2\pi(\frac{\xi}{\gamma}  + 1)+\frac{1}{\gamma}\big)=:b^{0.5}\gamma^{0.5}C B_a^{\rm m}\]
with $B_a^{\rm m}=\big(2\pi(\frac{\xi}{\gamma}  + 1)+\frac{1}{\gamma}\big)$.
Similarly, by (\ref{f_lj_C}), (\ref{eq:WP2}) and (\ref{eq:WP3}),  we choose
\[
 H(\partial_b;b) =b^{-1.5}\gamma^{0.5}C\big(0.5+2\pi(\xi+\gamma)+1\big)=:b^{-1.5}\gamma^{0.5}C B_b^{\rm m},
\]
and
\[H(\partial_c;b)=b^{-0.5}\gamma^{0.5}C2\pi\xi=:b^{-0.5}\gamma^{0.5}C B_c^{\rm m},\]
where  $B^m_b=\big(0.5+2\pi(\xi+\gamma)+1\big)$ and $B^m_c=2\pi\xi$.}
Moreover, by compositions of formulas (\ref{eq:WP1})--(\ref{eq:WP3}) for higher order derivatives, we choose
\[\begin{split}H(\partial_a\partial_b;b) & =b^{-0.5}\gamma^{0.5} C B_a^{\rm m}B_b^{\rm m}\\
 H(\partial_a\partial_c;b) & =b^{0.5}\gamma^{0.5}C B_a^{\rm m} B_c^{\rm m}\\
 H(\partial_b\partial_c;b) & =b^{-1.5}\gamma^{0.5}C B_b^{\rm m} B_c^{\rm m}\\
 H(\partial_a\partial_b\partial_c;b) & =b^{-0.5}\gamma^{0.5} C B_a^{\rm m} B_b^{\rm m} B_c^{\rm m}.
 \end{split}\]

Now, let us treat the two STFT parts 
 where we denote by $b_j$ either $b_0$ or $b_1$. We have
\[\begin{split} 
\partial_a \tau(a,b,c)f &  = \partial_a \Big( \sqrt{\frac{b_j}{\gamma}}e^{2i\pi(\frac{\xi}{\gamma} c b_j+ b)(x-a)}f\big(\frac{b_j}{\gamma}(x-a)\big)  \Big)\\
 & = -2i\pi(\frac{\xi}{\gamma} c b_j+ b)\tau(a,b,c)D_{0,0}f(x)  - \frac{b_j}{\gamma}\tau(a,b,c)D_{0,1}f(x).
 \end{split}\]
Moreover,
\[\begin{split}\partial_b\tau(a,b,c)f &= \partial_b \Big( \sqrt{\frac{b_j}{\gamma}}e^{2i\pi(\frac{\xi}{\gamma} c b_j+ b)(x-a)}f\big(\frac{b_j}{\gamma}(x-a)\big)  \Big)\\
 & =   2i\pi\frac{\gamma}{b_j} \tau(a,b,c)D_{1,0}f(x). 
 \end{split}\]
Last,
\[\begin{split}\partial_c\tau(a,b,c)f & = \partial_c \Big( \sqrt{\frac{b_j}{\gamma}}e^{2i\pi(\frac{\xi}{\gamma} c b_j+ b)(x-a)}f\big(\frac{b_j}{\gamma}(x-a)\big)  \Big)\\
& =  2i\pi\xi \tau(a,b,c)D_{1,0}f(x). 
\end{split}\]

Thus, as before, we choose for $0<b<b_0$ 
\[H(I;b)=b_0^{-0.5}\gamma^{0.5}C.\]
We moreover choose
\[H(\partial_a;b)=2\pi(\xi\gamma^{-0.5} c b_0^{0.5}+ b_0^{-0.5}\gamma^{0.5}b)C  + b_0^{0.5}\gamma^{-0.5}C \leq b_0^{-0.5}\gamma^{0.5}C B^{\rm l}_a,\]
where
 $B^{\rm l}_a = 2\pi(\xi\gamma^{-1}  b_0+ b_0)  + b_0\gamma^{-1}$. 
We choose
\[H(\partial_b;b)=2\pi\gamma^{1.5}b_0^{-1.5} C=:b_0^{-0.5}\gamma^{0.5}C B^{\rm l}_b,\]
with
$B^{\rm l}_b=2\pi\gamma  b_0^{-1}$.  
We  choose
\[H(\partial_c;b)=b_0^{-0.5}\gamma^{0.5}C 2\pi\xi =:b_0^{-0.5}\gamma^{0.5}C B^{\rm l}_c,\]
with
$ B^{\rm l}_c= 2\pi\xi$. 
Moreover, we define
\[\begin{split} 
H(\partial_a\partial_b;b)&=b_0^{-0.5}\gamma^{0.5}C B^{\rm l}_aB^{\rm l}_b,\\
H(\partial_a\partial_c;b)&=b_0^{-0.5}\gamma^{0.5} C B^{\rm l}_a B^{\rm l}_c,\\
H(\partial_b\partial_c;b)&=b_0^{-0.5}\gamma^{0.5}C B^{\rm l}_b B^{\rm l}_c,\\
H(\partial_a\partial_b\partial_c;b)&=b_0^{-0.5}\gamma^{0.5}C B^{\rm l}_a B^{\rm l}_b B^{\rm l}_c.
\end{split}\]

For high frequency STFT atoms, $b_1<b<L$, and we choose
\[\begin{split} 
H(I;b)&=b_1^{-0.5}\gamma^{0.5}C,\\
H(\partial_a;b)&=2\pi(\xi\gamma^{-0.5} c b_j^{0.5}+ b_j^{-0.5}\gamma^{0.5}b)C  + b_j^{0.5}\gamma^{-0.5}C <b_1^{-0.5}\gamma^{0.5} L C B^{\rm h}_a,
\end{split}\]
where $B^{\rm h}_a= 2\pi(\xi\gamma^{-1}  b_1 L^{-1}+ 1)  + b_1\gamma^{-1}$. 
We choose
\[H(\partial_b;b)=2\pi\gamma^{1.5}b_1^{-1.5} C=:b_1^{-0.5}\gamma^{0.5}C B^{\rm h}_b,\]
with 
$B^{\rm h}_b = 2\pi\gamma b_1^{-1}$. 
We choose
\[H(\partial_c;b)=b_1^{-0.5}C\gamma^{0.5} 2\pi\xi =:b_1^{-0.5}\gamma^{0.5}C B^{\rm h}_c,\]
with
$B^{\rm h}_c=2\pi\xi$. 
We moreover choose
\[\begin{split} 
H(\partial_a\partial_b;b)&=b_1^{-0.5}\gamma^{0.5}L C B^{\rm h}_aB^{\rm h}_b,\\
H(\partial_a\partial_c;b)&=b_1^{-0.5}\gamma^{0.5} L C B^{\rm h}_a B^{\rm h}_c,\\
H(\partial_b\partial_c;b)&=b_1^{-0.5}\gamma^{0.5}C B^{\rm h}_b B^{\rm h}_c,\\
H(\partial_a\partial_b\partial_c;b)&=b_1^{-0.5}\gamma^{0.5} L C B^{\rm h}_a B^{\rm h}_b B^{\rm h}_c.
\end{split}\]

It is now a matter of a direct calculation to show that
\begin{equation}
D=\sum_{\a\in \Gamma}\sum_{\b\in\NN_0:\b\leq \a} 
\left(
\begin{array}{c}
	\a \\
	\b
\end{array}
\right)
\iint_{([0,L]\times[0,1])|_{[\a]_{(b,c)}}}  H(\b;b) H(\a-\b;b) \ d{(b,c)}^{[\a]_{(b,c)}}= O(1).
\label{eq:dd5ge}
\end{equation}
Namely, $D$ is independent of the resolution $M$.
The main step in this calculation is to observe that the contribution to (\ref{eq:dd5ge}) due to the wavelet parts boils down to integration of a constant times $b^{-1}$, which is $O(\ln(\frac{b_1}{b_0})) = O(1)$, since $b_0=C_0L$ and $b_1=C_1 L$. 
   The dominant terms of the contribution to (\ref{eq:dd5ge}) due to the high frequency STFT is an integration over the integral $[C_2 L, L]$ of length $O(L)$, of a constant function $O(b_1^{-2}L)=O(L^{-1})$, and evaluations of a constant function of order $O(b_1^{-1}L) = O(1)$. Similarly, the contribution to (\ref{eq:dd5ge}) due to the low frequency STFT is $O(1)$.
\end{proof}

Proposition \ref{ClaimLTFT0} states that the QMC synthesis method of $V_f[s_M]$ has error of order
$O(\frac{M\log(N)^2}{N})$. 
%
%
%
In general, we consider phase space signal processing procedures that preserve the bounds $H(\alpha;b)$ of $V_{f}[s](a,b,c)$. For such procedures, the QMC LTFT method also has error rate
$O(\frac{M\log(N)^2}{N})$. One such example is multipliers (see Example \ref{ex:QMC_multiplier}). In the next subsection we study another example, namely, phase vocoder.

\subsection{Error analysis of QMC integer time dilation LTFT phase vocoder}

In integer time dilation phase vocoder, the QMC synthesis is computed for the dilated and phase corrected version of $V_f[s_M]$. The goal in this section is to illustrate that under certain assumptions the bounds $H(\a;b)$ of $\abs{\partial_{\a}V_f[s_M]}$ are preserved, up to a constant, under dilation and phase correction. Hence, the overall QMC phase vocoder method has error rate of $O(\frac{M\log(N)^2}{N})$. 
 The assumptions that we develop in this subsection are somewhat ad hoc, and in future work we will study general settings that satisfy these assumptions.

Let $S_M$ be normalized in $L_{\infty}(\RR)$. Let
 \[V_f[s_M](a,b,c)=S(a,b,c)=e^{i\theta(a,b,c)}R(a,b,c),\]
 for $\theta(a,b,c),R(a,b,c)\in\RR_{+}$,
 and consider the dilated signal in phase space
\[S_D(a,b,c) = e^{iD\theta(a/D,b,c)}R(a/D,b,c).\]
\textcolor{black}{Consider the LTFT bounds $\{H(\a;b)\}_{\a\in\Lambda}$ from Subsection \ref{Quasi-Monte Carlo LTFT error}. By (\ref{eq:bound_Vfs}), for every multi-index $\a\in\Lambda$ we have
\[\abs{\partial_{\a}S(a,b,c)}\leq H(\a;b).\]
In the next claim, we show that the first order $H$ bounds of $S(a,b,c)$ are preserved up to constant for $S_D(a,b,c)$}

\textcolor{black}{
\begin{claim}
 Under the above construction,
    \begin{align}
    \label{eq:C1}\abs{S_D(a,b,c)} & \leq H(I;b), \\
    \label{eq:C2}\abs{\partial_{a} S_D(a,b,c)}  & \leq (1+D^{-1})H(\partial_{a};b),\\
    \label{eq:C3}\abs{\partial_{b} S_D(a,b,c)}  & \leq (1+D)H(\partial_{b};b).
    \end{align}
\end{claim}
(\ref{eq:C3})
\begin{proof}
First, (\ref{eq:C1}) directly follows the change of variable $a/D\mapsto a$.
Next, we show that for any derivative $\partial_{\a}$ of first order, the bounds of $\abs{\partial_{\a}S_D(a,b,c)}$ can be taken as constant times $H(\a;b)$.
For that, note that  for every first order $\partial_{\a}$,
\[\partial_{\a} S=  i(\partial_{\a}\theta) e^{i\theta}R + e^{i\theta}(\partial_{\a}R).\]
Now, since $i\partial_{\a}\theta e^{i\theta}R$ is orthogonal to $e^{i\theta}\partial_{\a}R$ in the complex plane, we must have
\begin{equation}
\abs{\big(\partial_{\a}\theta(a,b,c)\big)R(a,b,c) } \leq \abs{\partial_{\a} S(a,b,c) }\leq H(\partial_{\a};b)
\label{eq:ra1}
\end{equation}
and
\begin{equation}
\abs{\partial_{\a}R(a,b,c)} \leq \abs{\partial_{\a} S(a,b,c) } \leq H(\partial_{\a};b).
\label{eq:ra2}
\end{equation}
We can hence bound $\abs{\partial_{\a}S_D}$ for all derivatives of order 1 using (\ref{eq:ra1}) and (\ref{eq:ra2}) as follows
\[\begin{split}
\abs{\partial_{a} S_D(a,b,c)} &= \abs{i [\partial_a \theta](a/D,b,c) e^{iD\theta(a/D,b,c)} R(a/D,b,c) +   \frac{1}{D}e^{iD\theta(a/D,b,c)}[\partial_{a} R](a/D,b,c)} \\
& \leq (1+D^{-1})H(\partial_{a};b) , 
\end{split}\]
\[\begin{split}
\abs{\partial_{b} S_D(a,b,c)}  & = \abs{i D\partial_b \theta(a/D,b,c) e^{iD\theta(a/D,b,c)} R(a/D,b,c) +   e^{iD\theta(a/D,b,c)}\partial_{b} R(a/D,b,c)}\\
 & \leq (1+D)H(\partial_{b};b),
\end{split}\]
and
\[\begin{split}
\abs{\partial_{c} S_D(a,b,c)} & = \abs{i D \partial_c \theta(a/D,b,c) e^{iD\theta(a/D,b,c)} R(a/D,b,c) +   e^{iD\theta(a/D,b,c)}\partial_{c} R(a/D,b,c)}\\
 & \leq (1+D)H(\partial_{c};b).
\end{split}\]
\end{proof}
}

\textcolor{black}{Obtaining bounds of the form $\abs{\partial_{\a}S_D(a,b,c)}\leq J H(\a;b)$ for higher order derivatives $\partial_{\a}$, and constants $J$, is more involved and requires some assumptions.} In the following discussion we motivate these assumptions by heuristic arguments.
Let us study as an example the term $\partial_a\partial_b S_D$. We have
\begin{equation}
\partial_{a}\partial_{b} S_D(a,b,c)= e^{iD\theta(a/D,b,c)}\Big(-D  X^{a,b}_1(a/D,b,c) + \frac{1}{D}X^{a,b}_2(a/D,b,c) +  i Y^{a,b}(a/D,b,c)\Big)
\label{eq:dfmmg8}
\end{equation}
where
\[\begin{split} 
X^{a,b}_1 &= \partial_a \theta\partial_b \theta R, \\
X^{a,b}_2 &= \partial_a\partial_{b} R,   \\
Y^{a,b} & =  \partial_a\partial_b \theta R  +  \partial_a \theta \partial_b R  +   \partial_b\theta\partial_{a} R .
\end{split}\]
Note that for $D=1$, by the fact that $S_1=S$, and by orthogonality in the complex plane, 
\begin{equation}
\abs{-X^{a,b}_1 + X^{a,b}_2 } \leq \abs{\partial_a\partial_b S}
\label{eq:4tw}
\end{equation}
and
\[\abs{Y^{a,b}} \leq \abs{\partial_a\partial_b S}.\]
Therefore, the ``imaginary'' term
$ \abs{Y^{a,b}(a/D,b,c)}$ of (\ref{eq:dfmmg8}) is bounded by $\abs{\partial_a\partial_b S}$ for any $D$.
To bound  the ``real'' term of (\ref{eq:dfmmg8}) by $C\abs{\partial_a\partial_b S}$ for some constant $C$ we need an assumption. 
 Note that (\ref{eq:4tw}) follows from orthogonality in the complex plane.
However, to bound the real term of (\ref{eq:dfmmg8}) we need to bound the terms $\abs{ X^{a,b}_1}$ and $\abs{X^{a,b}_2}$ separately.
If one of these terms is asymptotically larger than $\abs{\partial_a\partial_b S}$, then so must the other, since their sum has magnitude $\abs{\partial_a\partial_b S}$.
It is thus enough to assume that 
\begin{equation}
\abs{X^{a,b}_2} = \abs{\partial_b\partial_{a} R} \leq C\abs{\partial_a\partial_b S}
\label{eq:Asss1}
\end{equation}
for some constant $C$.
We justify this assumption heuristically as follows. If $X^{a,b}_1,X^{a,b}_2 \gg \abs{\partial_a\partial_b S}$, we must have
\begin{equation}
\frac{X^{a,b}_1}{X^{a,b}_2}\approx \pm 1.
\label{eqin2}
\end{equation}
Informally, the restriction (\ref{eqin2}) defines a subspace of co-dimension 1 in some space of functions, and a generic choice of $X^{a,b}_1$ and $X^{a,b}_2$ will typically not be in this subspace, since it is of measure zero. Of course, this argument is not mathematically rigorous, and is given here purely to inspire some intuition for Assumption (\ref{eq:Asss1}).
As a result of Assumption (\ref{eq:Asss1}), we must also have
\begin{equation}
\abs{X^{a,b}_1} \leq (C+1)\abs{\partial_a\partial_b S},
\label{eq:Asss15}
\end{equation}
so
\begin{equation}
\abs{-D  X^{a,b}_1(a/D,b,c) + \frac{1}{D}X^{a,b}_2(a/D,b,c) } \leq (C+1)(D+D^{-1})\abs{\partial_a\partial_b S},
\label{eq:ghr4}
\end{equation}
and hence
\begin{equation}
\abs{\partial_{a}\partial_{b} S_D(a,b,c)} \leq \big(1+(C+1)(D+D^{-1})\big)H(\partial_a\partial_b,b).
\label{eqdl:74f}
\end{equation}

 
A similar analysis for all other partial derivatives of order 2,3 in $\Lambda$ gives the following. For $\partial_a\partial_c$ we have
\begin{equation}
\partial_{a}\partial_{c} S_D(a,b,c)= e^{iD\theta(a/D,b,c)}\Big(-D  X^{a,c}_1(a/D,b,c) + \frac{1}{D}X^{a,c}_2(a/D,b,c) +  i Y^{a,c}(a/D,b,c)\Big)
\label{eq:dfmmg88}
\end{equation}
where
\[\begin{split} 
X^{a,c}_1 &= \partial_a \theta\partial_c \theta R, \\
 X^{a,c}_2&  = \partial_{a}\partial_c R,  \\
  Y^{a,c} & =  \partial_a \partial_c\theta R  +  \partial_a \theta \partial_c R  +   \partial_c\theta\partial_{a} R . 
  \end{split}\]
We assume 
\begin{equation}
\abs{X^{a,c}_2} = \abs{\partial_{a}\partial_c R} \leq C\abs{\partial_a\partial_c S},
\label{eq:Asss2}
\end{equation}
and obtain
\begin{equation}
\abs{\partial_{a}\partial_{c} S_D(a,b,c)} \leq \big(1+(C+1)(D+D^{-1})\big)H(\partial_a\partial_c;b).
\label{eqdl:74f2}
\end{equation}

For $\partial_b\partial_c$, we have
\begin{equation}
\partial_{b}\partial_{c} S_D(a,b,c)= e^{iD\theta(a/D,b,c)}\Big(-D^2  X^{b,c}_1(a/D,b,c) + X^{b,c}_2(a/D,b,c) +  iD Y^{b,c}(a/D,b,c)\Big)
\label{eq:dfmmg8888}
\end{equation}
where
\[\begin{split} X^{b,c}_1 &= \partial_b \theta\partial_c \theta R, \\
X^{b,c}_2 &= \partial_{b}\partial_c R,   \\
Y^{b,c}&=  \partial_b \partial_c\theta R  +  \partial_b \theta \partial_c R  +   \partial_c\theta\partial_{b} R . 
\end{split}\]
We assume 
\begin{equation}
\abs{X^{b,c}_2} = \abs{\partial_{b}\partial_c R} \leq C\abs{\partial_b\partial_c S},
\label{eq:Asss3}
\end{equation}
and obtain
\begin{equation}
\abs{\partial_{a}\partial_{c} S_D(a,b,c)} \leq \big(1+C(D^2+1)\big)H(\partial_b\partial_c;b).
\label{eqdl:74f21}
\end{equation}

For $\partial_a\partial_b\partial_c$, we have
\[\begin{split}  
\partial_a\partial_b\partial_c S_D(a,b,c) &= e^{iD\theta(a/D,b,c)}\Big(-D  X^{a,b,c}_1(a/D,b,c) + \frac{1}{D}X^{a,b,c}_2(a/D,b,c)\Big)\\
 & \quad + i e^{iD\theta(a/D,b,c)}\Big( Y^{a,b,c}_1(a/D,b,c) -  D^2 Y^{a,b,c}_2(a/D,b,c)\Big),
 \end{split}\]
where
\[\begin{split}X^{a,b,c}_1 & = \partial_b\partial_a \theta \partial_c\theta  R  +\partial_a \partial_c\theta \partial_b \theta  R  + \partial_a \theta \partial_b\partial_c \theta  R\\
& \quad  +\partial_a \theta \partial_b \theta  \partial_cR  +\partial_a \theta \partial_c \theta  \partial_b R +\partial_b\theta \partial_c \theta \partial_{a} R.  \\
X^{a,b,c}_2 & =  \partial_a\partial_b\partial_c R. \\
Y^{a,b,c}_1 & =\partial_c\partial_b\partial_a \theta R +\partial_b\partial_a \theta\partial_c R +\partial_a\partial_c \theta \partial_b R\\
& \quad +\partial_a \theta \partial_b\partial_c R +   \partial_b\partial_c\theta  \partial_{a} R   +  \partial_b\theta \partial_{a}\partial_c R +   \partial_c \theta\partial_b\partial_{a} R. \\
Y^{a,b,c}_2 & = \partial_a \theta\partial_b \theta\partial_c\theta R.
\end{split}\]
Here, we assume
\begin{equation}
\abs{X^{a,b,c}_2} = \abs{\partial_a\partial_{b}\partial_c R} \leq C\abs{\partial_a\partial_b\partial_c S},
\label{eq:Asss4}
\end{equation}
\begin{equation}
\quad \ \ 
\abs{ Y^{a,b,c}_2} = \abs{\partial_a \theta\partial_b \theta\partial_c\theta R} \leq C\abs{\partial_a\partial_b\partial_c S},
\label{eq:Asss5}
\end{equation}
and obtain
\begin{equation}
\abs{\partial_{a}\partial_b\partial_{c} S_D(a,b,c)} \leq (C+1)(D^2+D+1+D^{-1})H(\partial_b\partial_c;b).
\label{eqdl:74f22}
\end{equation}

As a result of Assumptions (\ref{eq:Asss1},\ref{eq:Asss2},\ref{eq:Asss3},\ref{eq:Asss4},\ref{eq:Asss5}), the QMC error of the phase vocoder method is
\[O(\frac{M\log(N)^2}{N}D^2)\]
%
%
Last, we note that Assumptions (\ref{eq:Asss1},\ref{eq:Asss2},\ref{eq:Asss3},\ref{eq:Asss4},\ref{eq:Asss5}) are satisfied for signals of the form $s(x)= e^{i\w_0 x}$, so they do not define the empty set. We leave finding spaces of signals satisfying Assumptions (\ref{eq:Asss1},\ref{eq:Asss2},\ref{eq:Asss3},\ref{eq:Asss4},\ref{eq:Asss5}) for future work.


%

\bibliographystyle{plain}	
\bibliography{ref1,ref2}

\appendix

\section{LTFT discretization}
\label{LTFT discretization}

We consider the following discretization of the LTFT, which generalizes the discretization of \cite{Ours1,Ours2}. For the general class of continuous signals we consider the Paley-Wiener space $\cS=PW(L)$ of signals $s\in L^2(\RR)$ with frequency support ${\rm supp}(\hs)\subset [0,L]$, where \textcolor{black}{$L/2>0$} is called the \emph{sample-rate}  of the signal. The discretization of $PW(L)$ is a sequence of finite dimensional subspaces $\{\cS_M\}_{M\in\ZZ}$ of $L^2(\RR)$, of dimension ${\rm dim}(\cS_M)=M$ for each $M$, where signals $s_M\in \cS_M$ have time supports ${\rm supp}(s_M)\subset [-M/L,M/L]$.
The spaces $\cS_M$ are chosen such that for any $s\in PW(L)$, there is a sequence of discrete signals $s_M\in \cS_M$ such that $\lim_{M\rightarrow\infty}\norm{s_M-s}=0$. Moreover, we choose the spaces $\cS_M$ to have most of the energy of $\hat{s}_M$, for any $s_M\in \cS_M$, concentrated about the band $[-L,L]$.
For example, in \cite{Ours1,Ours2}, $\cS_M$ is the space of trigonometric polynomials of order $M$ in $L^2[-M/L,M/L]$. We can also take, for example, $\cS_M$ as a space of linear splines supported in $[-M/L,M/L]$ with $M$ nodes at equidistant locations.

The LTFT of any signal $s_M\in \cS_M$ has most of its energy localized in phase space about the compact domain
\[G_M= [-M/L-S_0,M/L+S_0]\times [0,L]\times [0,1].\]
Indeed, most of the energy of $\hat{s}_M\in \cS_M$ is  concentrated about the band $[-L,L]$, the time support of $s_M$ are in $[-M/L,M/L]$, and the maximal support of LTFT atoms in $S_0$. 
This claim was rigorously formulated and proved in \cite{Ours2}. We hence restrict the phase space of the LTFT to $G_M$, calling the restricted system LTFT$^M$. We denote the synthesis operator of LTFT$^M$ by $V_f^{* M}$, namely,
\[V_f^{* M}F= \iint_{G_M} F(a,b,c)f_{a,b,c} dadbdc.\]

\subsection{The frame operator of discrete LTFT}

In \cite{Ours2} the frame operator of the LTFT was constructed. Here, even though the LTFT$^M$ system is not a frame on $L^2(\RR)$, \textcolor{black}{it is still customary to use} the term \emph{frame operator} for $S_f^M=V_f^{*M}V_f$. Next, we formulate the frame operator with respect to our parametrization of the LTFT, on the compact domain $G_M$ in phase space, and show how to efficiently compute it. Denote by $\mathbf{1}_Z$ the indicator function of the set $Z$.

\begin{proposition}[The frame operator of LTFT$^M$]
\label{LTFTM_Sf}
Let
\begin{equation}
P_0=\abs{\hf\big(\frac{\gamma}{b_0}(\cdot)\big)}^2 * \mathbf{1}_{[0,\frac{\xi b_0}{\gamma}]}, \quad 
 P_1= \frac{\gamma}{\xi}\abs{\hf\Big((\cdot) - \gamma \Big) }^2 \mathbf{1}_{[0,\xi]} , \quad
 P_2 = \frac{\gamma^2}{b_1^2\xi}\ \abs{\hf\big(\frac{\gamma}{b_1}(\cdot)\big)}^2 * \mathbf{1}_{[0,\frac{\xi b_1}{\gamma}]},
    \label{EQ:86}
\end{equation}
\begin{equation}
Q_0=P_0*\mathbf{1}_{[0,b_0]} , \quad
 Q_2= P_2*\mathbf{1}_{[b_1,L]},
    \label{EQ:5n6s}
\end{equation}
and $Q_1:\RR\rightarrow\RR$ be defined by
\begin{equation}
  Q_1(\w) = \int_{\frac{\gamma \w}{b_1}}^{\frac{\gamma \w}{b_0}} \frac{1}{q}P_1(q) dq.
  \label{S_f_CWT0}
\end{equation}
Let $H=Q_0+Q_1+Q_2$. 
The frame operator $S_f^M=V_f^{*M}V_f$ of LTFT$^M$ is given by
\[\cF S_f^M\cF^{*}\hs(\w) = H(\w)\hs(\w).\]
\end{proposition}
\textcolor{black}{The proof of this proposition is given below.}
Note that all functions $P_j,Q_j$, $j=1,2,3$, of (\ref{EQ:86})--(\ref{S_f_CWT0}) can be computed in $O(M\log(M))$ operations, where $M$ is the number of frequency samples in the discrete computation. For (\ref{S_f_CWT0}), all values  $Q_1(\w)$ can be computed by the values of $P_1$  in $O(M)$ operations. Indeed, the integration in (\ref{S_f_CWT0}) for one value of $\w$ in the grid, can be computed using the value of (\ref{S_f_CWT0}) on a neighboring $\w'$, with the addition and subtraction of $O(1)$ values due to the difference in integration domains. Thus, the overall computational complexity for computing $H$ is $O(M\log(M))$, and $S_f$ can be computed in pre-processing once and for all.

\textcolor{black}{To prove of Proposition \ref{LTFTM_Sf}, we start with a lemma.}
 Denote by $\hat{\tau}$ the frequency representation of $\tau$, namely,
\[\hat{\tau}(a,b,c) = \cF\tau(a,b,c) \cF^*.\]
\begin{lemma}[Frequency representation of the LTFT]
Let $f_{a,b,c}$ be the LTFT atoms. Then
\begin{equation}
\cF f_{a,b,c}(\w)=\hat{f}_{a,b,c}(x)  = [\hat{\tau}(a,b,c)\hat{f}](\w)
= \left\{
\begin{array}{ccc}
	\sqrt{\frac{\gamma}{b_0}}e^{-2\pi i a \w}\hf\big(\frac{\gamma}{b_0} (\w-\frac{\xi}{\gamma}c b_0-b) \big) & {\rm if} & b<b_0 \\
	\sqrt{\frac{\gamma}{b}}e^{-2\pi i a \w}\hf\Big(\frac{\gamma}{b} \big(\w-(\frac{\xi}{\gamma}c +1)b\big) \Big)   & {\rm if} & b_0<b<b_1 \\
\sqrt{\frac{\gamma}{b_1}}e^{-2\pi i a \w}\hf\big(\frac{\gamma}{b_1} (\w-\frac{\xi}{\gamma}c b_1-b) \big)  & {\rm if} &  b>b_1.
\end{array}
\right.
\label{eq:LTFT_atom_F}
\end{equation}
\end{lemma}

\begin{proof}
By (\ref{eq:LTFT_atomOP}) and Lemma \ref{Transform_lemma}
\begin{equation}
 \hat{\tau}(a,b,c)
= \left\{
\begin{array}{ccc}
	\mathcal{M}(-a)\mathcal{T}(\frac{\xi}{\gamma}c b_0+b)\mathcal{D}(\frac{b_0}{\gamma}) & {\rm if} & b<b_0 \\
	  \mathcal{M}(-a)\mathcal{T}\big((\frac{\xi}{\gamma}c +1)b\big)\mathcal{D}(\frac{b}{\gamma}) & {\rm if} & b_0<b<b_1 \\
\mathcal{M}(-a)\mathcal{T}(\frac{\xi}{\gamma}c b_1+b)\mathcal{D}(\frac{b_1}{\gamma})  & {\rm if} &  b>b_1,
\end{array}
\right.
\label{eq:LTFT_atom_F_OP}
\end{equation}
which gives (\ref{eq:LTFT_atom_F}).
\end{proof}

\begin{proof}[Proof of Proposition \ref{LTFTM_Sf}]

\textcolor{black}{We define the function
\begin{equation}
    \label{eq:F01}
    F(b,c,\cdot):= \cF \Big(\int_a V_f[s] f_{a,b,c}  da\Big),
\end{equation}
and note that by (\ref{eq:M_Sn})
\begin{equation}
  \cF V_f^{*M}V_f[s](\w) = \int_0^1\int_0^L F(b,c,\w) dbdc,
  \label{EQ_S_fall0}
\end{equation}}
\textcolor{black}{
In all cases of (\ref{eq:LTFT_atom_F}) and (\ref{eq:LTFT_atomOP}), by Definition \ref{def:TMD} and Lemma \ref{Transform_lemma}, the LTFT atom can be written as
\begin{equation}
    \label{eq:LTFT_M}
    \hat{f}_{a,b,c}(\w) = \mathcal{M}(-a)\hf_{0,b,c}(\w) = e^{-2\pi i a \w} \hf_{0,b,c}(\w),
\end{equation}
\begin{equation}
    \label{eq:LTFT_M22}
    f_{a,b,c}(x) = \mathcal{T}(a)f_{0,b,c}(x) = f_{0,b,c}(x-a).
\end{equation}
Hence, by (\ref{eq:M_An}) and Lemma \ref{Transform_lemma}, we have
\[
\begin{split}
   F(b,c,\w) &  =  \int_a \int_x s(x)\overline{ f_{0,b,c}(x-a)} dx e^{-2\pi i a \w} \hf_{0,b,c}(\w) da \\
   & = \int_x s(x) \int_a \overline{ f_{0,b,c}(x-a)}e^{-2\pi i a \w} da dx  \hf_{0,b,c}(\w)  \\
   & = \int_x s(x)  \overline{\int_a \mathcal{T}_x f_{0,b,c}(-a)e^{2\pi i a \w} da }dx  \hf_{0,b,c}(\w) \\
   & = \int_x s(x)  \overline{\mathcal{D}(x) \hf_{0,b,c}(\w)}dx  \hf_{0,b,c}(\w) \\
   & = \int_x s(x)  e^{-2\pi i x \w}dx  \abs{\hf_{0,b,c}(\w)}^2  =\hs(\w)  \abs{\hf_{0,b,c}(\w)}^2 .
\end{split}
\]}

Let us split (\ref{EQ_S_fall0}) to the three sub-domains in phase space
\begin{equation}
\begin{split}
  &\cF V_f^{*M}V_f(s)(\w) = \\
   &\Big(\int_0^1\int_0^{b_0} F_0(b,c,\w) dbdc + \int_0^1\int_{b_0}^{b_1} F_1(b,c,\w) dbdc + \int_0^1\int_{b_1}^L F_2(b,c,\w) dbdc\Big)\hat{s}(\w),
  \end{split}
  \label{EQ_S_fall01}
\end{equation}
with $F_0,F_1,F_2$ the restrictions of $F$ to $b\in[0,b_0]$, $b\in (b_0,b_1)$, and $b\in [b_1,L]$ respectively.
We study separately the three components of (\ref{EQ_S_fall01}).

For low frequencies,
\[\hf_{0,b,c}(\w) = \sqrt{\frac{\gamma}{b_0}}e^{-2\pi i a \w}\hf\big(\frac{\gamma}{b_0} (\w-\frac{\xi}{\gamma}c b_0-b) \big) \]
so
\[\int_0^1 \abs{\hf_{0,b,c}(\w)}^2 dc= \int_0^1  \frac{\gamma}{b_0}\abs{\hf\big( \frac{\gamma}{b_0}\w-\xi c -\frac{\gamma}{b_0} b \big)}^2 dc.\]
By changing variable $\xi c = \frac{\gamma}{b_0} z$ we have
\[\begin{split}
    \int_0^1 \abs{\hf_{0,b,c}(\w)}^2 dc= & \frac{\gamma^2}{b_0^2\xi}\int_0^{\frac{\xi b_0}{\gamma}}  \abs{\hf\big( \frac{\gamma}{b_0}(\w-b- z)  \big)}^2 dz\\
    = & \frac{\gamma^2}{b_0^2\xi}\Big[\abs{\hf\big(\frac{\gamma}{b_0}(\cdot)\big)}^2 * \mathbf{1}_{[0,\frac{\xi b_0}{\gamma}]}\Big](\w-b) =: P_0(\w-b).
\end{split}
\]
Thus
\[Q_0(\w):=\int_0^1\int_0^{b_0} F_0(b,c,\w) dbdc
    =\int_0^{b_0} P_0(\w-b)db = [P_0*\mathbf{1}_{[0,b_0]}](\w).
\]
Similarly, we derive $P_2$ and $Q_2$ of (\ref{EQ:86}) and  (\ref{EQ:5n6s}).

Last, for middle frequencies we have
\[\hf_{0,b,c}(\w) =  \sqrt{\frac{\gamma}{b}}\hf\Big(\frac{\gamma}{b} \big(\w-(\frac{\xi}{\gamma}c +1)b\big) \Big). \]
We compute
\[\int_0^1 \int_{b_0}^{b_1}\abs{\hf_{0,b,c}(\w) }^2  db dc=  \int_0^1 \int_{b_0}^{b_1} \frac{\gamma}{b}\abs{\hf\Big( \frac{\gamma\w}{b}-\xi c -\gamma\big) \Big)}^2 db dc.\]
By the change of variable $\frac{\gamma \w}{b}=q$, $db = -\frac{\gamma\w}{q^2}dq$, we have 
\[\int_0^1 \int_{b_0}^{b_1}\abs{\hf_{0,b,c}(\w) }^2  db dc=  \int_0^1 \int_{\frac{\gamma \w}{b_1}}^{\frac{\gamma \w}{b_0}} \frac{\gamma}{q}\abs{\hf\Big( q-\xi c -\gamma \Big)}^2 db dc.\]
Moreover, by changing variable $\xi c= z$ we define 
\[\int_0^{\xi} \frac{\gamma}{\xi} \abs{\hf\Big( q-z -\gamma\Big) }^2 dz = \Big[\frac{\gamma}{\xi}\abs{\hf\Big((\cdot) - \gamma \Big) }^2 \mathbf{1}_{[0,\xi]}\Big](q) = : P_1(q).\]
Hence, we define (\ref{S_f_CWT0}) by
\[
  Q_1(\w) = \int_{\frac{\gamma \w}{b_1}}^{\frac{\gamma \w}{b_0}} \frac{1}{q}P_1(q) dq
\]
%
which means that the frame operator $S_f^M=V_f^{*M}V_f$ is given by
\[\cF S_f^M\cF^{*}\hs(\w) = \big(Q_0(\w)+Q_1(\w)+Q_2(\w)\big)\hs(\w).\]

\end{proof}

\section{Discrepancy of DWT grids}
\label{Discrepancy of DWT grids}

In this appendix we compute the discrepancy of the discrete wavelet transform (DWT) grid. We note that the discrepancy of the standard grid of STFT is well-known to be $O(N^{-1/2})$, where $N$ is the number of grid points.

\subsection{A DWT construction}
\label{A DWT construction}

Consider the following setting, similar to Morlet wavelets \cite{wavelet_tour}. Let $M$ be the resolution of the discrete signals,  supported in the time interval $[-M/L,M/L]$ and in the frequency interval $[0,L]$. Consider a window $f$, with $\hf$ centered about $\w=0$. Suppose that $\hf$ is concentrated on the interval $(-0.5,0.5)$. Denote the number of oscillations in the mother wavelet by $\gamma$, and define the mother wavelet $h(x)=e^{2\pi i\gamma x}f(x)$. Suppose that $h$ is admissible (satisfying (\ref{eq:addmiss})). In the DWT grid we consider dilation samples of the form $\{r^k\}_{k\in\ZZ}$, for some $r>1$.
To guarantee that the wavelet transform can be stably reconstructed,  we require that $\abs{\hh(\w)}^2$ and $\abs{\hh(r^{-1}\w)}^2$ are concentrated on intersecting intervals \cite{wavelet_tour}. Thus, we demand
\[\big(\gamma-0.5,\gamma+0.5\big)\cap \big((\gamma-0.5)r,(\gamma+0.5)r\big) \neq \emptyset.\]
 Namely,
$1<r<\frac{\gamma+0.5}{\gamma-0.5}$. 
We thus consider $r$ of the form
\[1<r=\Big(\frac{\gamma+0.5}{\gamma-0.5}\Big)^q\]
with $0<q<1$.

For each sample scale $r^k$, we consider $r^k\frac{M}{L}p$ time samples in a uniform grid, where $p>0$ is a constant that controls the time spacing. 

\subsection{DWT sample set size}
\label{DWT sample set size}

Let us estimate the size of the DWT grid. Suppose we wish to represent the signal in the frequency band $[b_0,L]$. The largest dilation in the discrete transform is the smallest $K_1$ satisfying
\[r^{K_1}(\gamma-0.5)\geq L,\]
or
\[K_1\ln(r)\geq \ln(\frac{L}{\gamma-0.5}).\]
This guarantees that the whole frequency interval $[b_0,L]$ is covered by discrete wavelets. 
For the asymptotic analysis, we assume without loss of generality equality
\[K_1 =  \frac{\ln(\frac{L}{\gamma-0.5})}{\ln(r)}.\]
Similarly, the smallest $k$ in the transform is
\[K_0 =  \frac{\ln(\frac{b_0}{\gamma+0.5})}{\ln(r)}.\]

The number of time samples for each $k$ is $r^k\frac{M}{L}p$. Let us estimate the total number of sample points by the continuous integral
\[N=\int_{K_0}^{K_1} r^k\frac{M}{L}p dk= \frac{M}{L\ln(r)}p (r^{K_1}-r^{K_0})\]
\[=\frac{Mp}{qL\ln\big(\frac{\gamma+0.5}{\gamma-0.5}\big)} \Big(\frac{L}{\gamma-0.5}-\frac{b_0}{\gamma+0.5}\Big)\]
\begin{equation}
\label{DWT_N_est}
    \geq H(\gamma,p,q) (L-b_0),
\end{equation}
where
\[H(\gamma,p,q) = \frac{Mp}{qL\ln\big(\frac{\gamma+0.5}{\gamma-0.5}\big)}\frac{1}{\gamma+0.5}.\]

To allow the DWT grid to become finer, we consider any combination of $M\rightarrow\infty$, $p\rightarrow\infty$,  $q\rightarrow 0$, and fixed $\gamma$. Equivalently, we may consider varying $M$, $r$ and $p$. By (\ref{DWT_N_est}) and by the fact that $q$ is proportional to $\ln(r)$, if $L\gg b_0$, we have approximately,
\begin{equation}
N\approx C\frac{Mp}{\ln(r)},
    \label{DWT_DISC1}
\end{equation}
for some constant $C$ that depends on $\gamma$.
We denote the resulting DWT grid, having $N$ sample points, by $\mathcal{Q}_{N;p,r}$. \textcolor{black}{In Figure \ref{fig:DWT_g} we compare $\mathcal{Q}_{N;p,r}$ for different choices of $p$ and $r$.} 
\begin{figure}[!ht]
\centering
\includegraphics[width=0.8\linewidth]{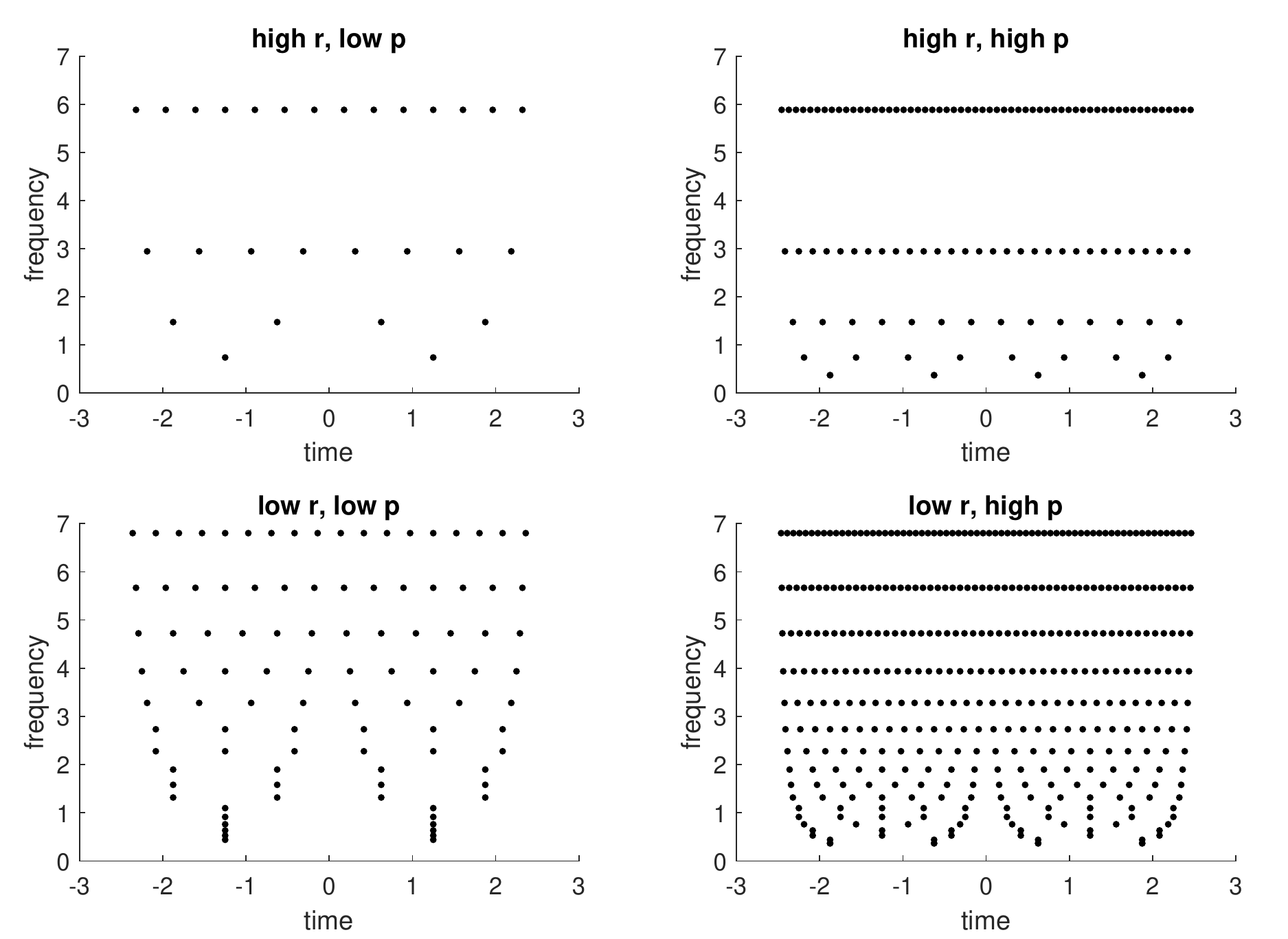}
\caption{\textcolor{black}{Comparison of different DWT grids $\mathcal{Q}_{N;p,r}$. The values $r^k$ are plotted along the frequency axis. For all four grids, we set $L=4,$ $\gamma=1,$ $b_0=0.5$ and $M=5$.
Top left: $r=2$ and $p=1$.  
Top right: $r=2$ and $p=4$. 
Bottom left: $r=1.2$ and $p=1$. 
Bottom right: $r=1.2$ and $p=4$. }
}
\label{fig:DWT_g}
\end{figure}


%
%

\subsection{DWT sample set discrepancy}

In this subsection we prove that the discrepancy of the DWT grid is sub-optimal.
In the following analysis we omit constants that are unchanged when the DWT grid becomes finer, e.g., $\gamma$. This does not affect the asymptotic analysis of the discrepancy bound.

\begin{claim}
Consider the DWT grid as defined in Subsections \ref{A DWT construction} and \ref{DWT sample set size}.  
Then, for large enough $q$, the star discrepancy of the DWT grid satisfies
\[D_N^*\mathcal{Q}_{N;q,r} \geq C'\frac{1}{\sqrt{N}}\]
for some $C'$ that does not depend on $N,q$ and $r$.
\end{claim}

\begin{proof}
We bound the discrepancy of the DWT sample set from below by constructing two rectangles that do not intersect the DWT samples. 

The first rectangle is supported in the time axis in $[-M/L,M/L]$, and in frequency it is supported in the last frequency gap, that has length of order
$L(1-r^{-1})$.
When scaling phase space to $[0,1]^2$, the area of this rectangle becomes
\[B_1=1-r^{-1}.\]
Next, we construct the second rectangle.
We take the bottom-left corner at time-frequency $(0,0)$. The top of the rectangle has  frequency coordinate $r^k\gamma$ for some $k$. 
The number of time samples of the DWT grid at frequency $k$ is $r^k\frac{M}{L}p$ over the time interval of length $\frac{M}{L}$. Hence, to guarantee that the rectangle does not intersect the CWT grid, we take the width of the rectangle as the time spacing $\frac{1}{r^k p}$. The area of this rectangle is
$\frac{r^k\gamma}{r^k p}$, which is or order of $p^{-1}$ since $\gamma$ is fixed in the asymptotic analysis of the DWT grid.
After rescaling phase space, the area is
$\frac{1}{M p}$.

Now, by (\ref{DWT_DISC1}) 
\[M\approx \frac{N\ln(r)}{p C}. \]
Ignoring constants, the area of the second rectangle, when phase space is scaled to $[0,1]^2$, is
\[B_2=\frac{1}{N \ln(r)}.\]
Denote $r=e^v$, and note that for $r$ close to $1$ (small $q$) we have the areas
\[B_1= 1- e^{-v}\approx v \]
\[B_2= \frac{1}{N v}.\]
Combining the bounds $B_1$ and $B_2$, the discrepancy of the DWT grid is bounded from below by
\[\max\{v,\frac{1}{Nv}\}=\left\{
\begin{array}{ccc}
	v & , & v\geq \frac{1}{\sqrt{N}}\\
	\frac{1}{Nv} & , & v\leq \frac{1}{\sqrt{N}}
\end{array}
\right.
\geq \frac{1}{\sqrt{N}}.\]
This shows that the discrepancy of the DWT grid is more than constant times $N^{-0.5}$.
\end{proof}

\section{Time-frequency tessellation in CWT and LTFT analysis}
\label{Time-frequency tessellation in CWT and LTFT analysis}

We present in this subsection an analogous notion to Heisenberg boxes in CWT analysis.  The essential domain covered by a wavelet time-frequency kernel $V_f[f_{a,b}]$, centered at $(a,b)$, is funnel shaped, and we thus call it a \emph{wavelet funnel}. For a fixed translation-dilation $(a',b')$, the inner product $V_f[f_{a,b}](a',b')=\ip{\mathcal{T}(a)\mathcal{D}(b)f}{\mathcal{T}(a')\mathcal{D}(b^{\prime -1})f}$ is localized for each fixed $b'$ and variable $a'$, at a time interval about $a$  of length $\frac{\gamma}{b}$, where $\gamma$ is the number of oscillations in the mother wavelet. The dilation parameters $b'$ are localized in the frequency interval $[\frac{b}{1+\gamma^{-1}},\frac{b}{1-\gamma^{-1}}]$. We thus define the wavelet funnel ${\cal W}(a,b)$ as the domain of the $(a',b')$ points which are confined between the curves 
\[a'=a\pm\frac{\gamma}{b'},\quad b'= \frac{b}{1\pm\gamma^{-1}}.\]

It is easy to see that that the characteristic function of ${\cal W}$ satisfies
\[\mathbf{1}_{{\cal W}(a,b)}(a',b') = \mathbf{1}_{{\cal H}(a',b')}(a,b),\]
where ${\cal H}(a',b')$ is the Heisenberg box centered at $(a',b')$, with time side $[a'-\frac{\gamma}{b'},a'+\frac{\gamma}{b'}]$ and frequency side $[b'-\frac{b'}{\gamma},b'+\frac{b'}{\gamma}]$. We thus call the wavelet funnels ${\cal W}$ the \emph{adjoint} of the Heisenberg boxes ${\cal H}$. 

The wavelet funnel represents the domain in which the wavelet kernel $V_f(f_{a,b})$ is concentrated. Hence, intuitively, a good CWT discretization is one for which the wavelet funnels centered at the sample points cover approximately uniformly the time-frequency plane. Namely, for a discretization  $\{(a_n,b_n)\}_{n=1}^N$, we would like to have for any $(a',b')$ in the restricted phase space $G_M$
\begin{equation}
\frac{M}{N}\sum_{n=1}^N\mathbf{1}_{{\cal W}(a_n,b_n)}(a',b')\approx C
\label{eq:rrtry}
\end{equation}
for some constant $C$ which is independent of $(a',b')$.
In the continuous limit we require
\[ \iint\mathbf{1}_{{\cal W}(a_n,b_n)}(a',b')dadb= C\]
for $(a',b')$ in some large time-frequency domain $G_N$.
We estimate the left-hand-side of (\ref{eq:rrtry}) using the discrepancy as follows. By choosing $C$ equal to the volume of the Heisenberg box $\mu({\cal H}(a',b'))$,
\[\abs{\frac{M}{N}\sum_{n=1}^N\mathbf{1}_{{\cal W}(a_n,b_n)}(a',b')- C} = \abs{\frac{M}{N}\sum_{n=1}^N\mathbf{1}_{{\cal H}(a',b')}(a_n,b_n) -\mu({\cal H}(a',b'))} \leq M D_N\big(\{(a_n,b_n)\}_{n=1}^N\big),\]
where $D_N$ is the discrepancy (\ref{eq:disc}).
In this sense, a low discrepancy point set is an optimal CWT sampling set. Hence, (\ref{eq:rrtry}) is true up to an error of $O(\frac{M\log^k(N)}{N})$, where $k=1$ for a low discrepancy point set, and $k=2$ for a low discrepancy sequence.

We can formulate an equivalent analysis for the middle frequency atoms of the LTFT. Here, we define the LTFT funnel, centered at $(a,b,c)$, as the domain $\mathcal{W}$ of $(a',b',c')$ points which are confined by the surfaces
\[c'=c\pm \nu\]
\[b'=\frac{b}{a\pm\frac{1}{c'}}\]
\[a' = a \pm \frac{c'}{b'}.\]
Here, $\nu$ is a constant that represents the range of modulations $c'\in[-\nu,\nu]$ of LTFT atoms $f_{a,b,c}$, for which $f_{a,b,(c+ c')}$ has significant correlation with $f_{a,b,c}$. 
It is easy to see that $\mathbf{1}_{\mathcal{W}(a,b,c)}(a',b',c') = \mathbf{1}_{\mathcal{H}(a',b',c')}(a,b,c)$, where $\mathcal{H}(a',b',c')$ is the 3D rectangle defined by
\[a'-\frac{c'}{b'} < a < a'+\frac{c'}{b'}\]
\[b'-\frac{b'}{c'} < b < b'+\frac{b'}{c'} \]
\[c'-\nu<c<c'+\nu. \]
Now, similarly to the CWT case, we can show that the funnels of the LTFT middle atoms cover the middle frequencies approximately uniformly. Moreover, the upper and lower frequency atoms of the LTFT are STFT atoms are represented by Heisenberg boxes, which cover the high and low frequency bands approximately uniformly as well. 

\subsection*{Acknowledgements}

R.L. acknowledges support by the DFG SPP 1798 “Compressed Sensing
in Information Processing” through Project Massive MIMO-II.


G.K. acknowledges support by the Deutsche Forschungsgemeinschaft (DFG)
through Project KU 1446/21-2 within SPP 1798.

H.A. acknowledges support by US-Israel Binational Science Foundation grant 2017698, and Israeli Science Foundation grant 1272/17.






\end{document}